\documentclass{amsart}
\usepackage{anysize}
\usepackage{xcolor}
\usepackage{hyperref}
\usepackage{amsmath}
\usepackage{amssymb}
\usepackage{amsthm}
\usepackage{bbm}
\usepackage{dsfont}
\usepackage{url}
\usepackage{geometry}
\usepackage{textpos}
\usepackage{enumerate}
\usepackage{graphicx}
\usepackage{subcaption}
\usepackage{subfloat}
\usepackage[bb=boondox]{mathalpha}
\usepackage{amsfonts}
\usepackage{mathtools}
\usepackage{hyperref}

\newtheorem{theorem}{Theorem}[section]
\newtheorem{lemma}[theorem]{Lemma}

\newtheorem{proposition}[theorem]{Proposition}
\newtheorem{corollary}[theorem]{Corollary}
\theoremstyle{remark}
\newtheorem{remark}[theorem]{Remark}
\theoremstyle{definition}

\newcommand{\sd}{\Sigma \Delta}

\newcommand{\R}{\mathbb{R}}
\newcommand{\C}{\mathbb{C}}
\newcommand{\Z}{\mathbb{Z}}
\newcommand{\one}{\scalebox{1}{$\mathbbm{1}$}}
\usepackage{cite,fullpage}

\hypersetup{
  pdftitle={Theory and algorithms for compressive data acquisition under practical constraints},
  pdfauthor={Kateryna Melnykova, {\"O}zg{\"u}r Yilmaz},
  pdfkeywords={quantization, MSQ, PCM, frame theory, random frames, compressed sensing, sparse recovery, two-stage recovery, projected back projection, PBP, WNH, WNA}
}
\begin{document}
\title{Memoryless scalar quantization for random frames}
\author{Kateryna Melnykova, \"{O}zg\"{u}r Yilmaz}
\newcommand{\OYnote}[1]{\textcolor{red}{[{\em {\bf **OY Note:} #1}]}}
\newcommand{\KMnote}[1]{\textcolor{blue}{[{\em {\bf **KM Note:} #1}]}}
\newcommand{\OYnotee}[1]{\textcolor{cyan}{[{\em {\bf **OY NoteNew:} #1}]}}
\newcommand{\KMnotee}[1]{\textcolor{brown}{[{\em {\bf **KM NoteNew:} #1}]}}
\newcommand{\OYnotef}[1]{\textcolor{magenta}{[{\em {\bf **OY NoteNew2:} #1}]}}

\begin{abstract}
Memoryless scalar quantization (MSQ) is a common technique to quantize frame coefficients of signals (which are used as a model for generalized linear samples), making them compatible with our digital technology. The process of quantization is generally not invertible, and thus one can only recover an approximation to the original signal from its quantized coefficients. The non-linear nature of quantization makes the analysis of the corresponding approximation error challenging, often resulting in the use of a simplifying assumption, called the ``white noise hypothesis'' (WNH) that simplifies this analysis. However, the WNH is known to be not rigorous and, at least in certain cases, not valid.

Given a fixed, deterministic signal,  we assume that we use a random frame, whose analysis matrix has independent isotropic sub-Gaussian rows, to collect the measurements, which are consecutively quantized via MSQ. For this setting, numerically observed decay rate seems to agree with the prediction by the WNH. We rigorously establish sharp non-asymptotic error bounds without using the WNH that explain the observed decay rate. Furthermore, we show that the reconstruction error does not necessarily diminish as redundancy increases. Specifically, for Gaussian random frames, we provide sharp non-asymptotic upper and lower error bounds; these bounds asymptotically converge to (small) non-zero constant. For other random frames, we show that the error bound does not diminish to zero asymptotically (under the condition that the dimension of the signal goes to infinity as well).

We also extend this approach to the compressed sensing setting, obtaining rigorous error bounds that agree with empirical observations, again, without resorting to the WNH. 
\end{abstract}
\maketitle
\section{Introduction}
We consider the problem of quantizing generalized linear measurements of finite dimensional signals, both in the overdetermined (frame theory) setting and in the underdetermined (compressed sensing) setting. We focus on memoryless scalar quantization (MSQ) which is arguably the most intuitive quantization strategy and can be easily implemented in practice. To be more specific, let $q_\delta:\R\mapsto \delta\Z$ be the {\it uniform scalar quantizer with stepsize $\delta>0$}, where $q_\delta(x):=n\delta$ if $x\in (n\delta-\delta/2,n\delta+\delta/2]$. The associated \emph{uniform MSQ with stepsize $\delta$} is $Q^{\rm MSQ}_\delta:\R^m\mapsto  \left(\delta\Z\right)^m$ such that 
the $j$th component of $Q^{\rm MSQ}_\delta(u)$ satisfies $(Q^{\rm MSQ}_\delta(u))_j=q_\delta(u_j)$. When there is no ambiguity we shall use $Q$ to denote $Q^{\rm MSQ}_\delta$. 

Suppose that the signal to be acquired is $x\in\mathbb{R}^k$, and let us restrict our attention to the frame theory setting. Accordingly, let $E\in\mathbb{R}^{m\times k}$ be (the analysis matrix of) a frame for $\R^k$ (where $m\ge k$) and denote the generalized samples, i.e., the frame coefficients, of $x$ by $y=Ex$. Since $E$ is a frame, $x$ can be recovered exactly from $y$ via ${\widetilde{E}y=\widetilde{E}Ex}=x$, where $\widetilde{E}$ is any left inverse, i.e., dual, of $E$. On the other hand, if the frame coefficients are quantized, say using $Q=Q^{\rm MSQ}_\delta$, then the recovery is, in general, inexact and results in an approximation $\widetilde{E}Q(Ex)$ with non-zero \emph{reconstruction error} $\|x-\widetilde{E}Q(Ex)\|$.

The non-linear nature of the quantizer $Q$ makes the analysis of the reconstruction error challenging. Accordingly, even though the error expression is deterministic once $x$ and $E$ are fixed, it is customary to rely on the so-called {\it white noise hypothesis (WNH)} for its analysis. The WNH, originally proposed by Bennett in \cite{BennettWNH} in the context of scalar quantization, asserts that, for a fixed deterministic measurement matrix $E$, if the signal $x$ is random, then the \emph{quantization error} ${Ex-Q(Ex)}$ is a random vector with independent, identically distributed (i.i.d.) entries that are uniformly distributed over $(-\delta/2,\delta/2]$. Unfortunately, the WNH is not rigorous and does not hold, at least in certain cases, e.g., \cite{Viswanathan,Jimenez, benedetto2006sigma}.

In this paper, we consider the reconstruction error when the generalized samples of a fixed, deterministic signal $x$ are obtained via a sub-Gaussian random frame $E$. (This choice is motivated by the use of such random matrices in compressed sensing.) While an appropriately modified WNH simplifies the error analysis, as we show in Section \ref{sec:MSQFramesReview}, this hypothesis is also non-rigorous and provably not valid, for example, when $E$ is a Gaussian random frame. Instead, we provide a rigorous methodology to establish rigorous error bounds in this setting. Furthermore, our approach can be generalized to compressed sensing setting, yielding error bounds that match with empirical observations.

Throughout, $\|\cdot\|$ refers to $\ell_2$-norm  and $C,c_1,c_2,...$ indicate constants. The Moore-Penrose pseudoinverse of a matrix $E$ with full column-rank is  denoted by $E^\dagger$ and given by $E^\dagger=(E^TE)^{-1}E^T$. For $x\in \R^N$, $x_T \in \R^k$ is the restriction of $x$ to $T$ when $T\subset\{1,2,...,N\}$ with $\# T=k$. For $\Phi\in \R^{m\times N}$, $\Phi_T$ denotes the $m\times k$ submatrix of $\Phi$ obtained by its restriction to the columns indexed by $T$. 

For a sub-Gaussian random variable $z$, $\|z\|_{\psi_2}=\sup_{p\geq 1}p^{1/2}\left(\mathbb{E}|z|^p\right)^{1/p}$ is the sub-Gaussian norm of $z$.  Note that $\|\cdot\|_{\psi_2}$ is a norm on the Banach space of sub-Gaussian random variables -- see, e.g., \cite{buldygin1980sub}.

A random vector $Y\in\mathbb{R}^n$ is called sub-Gaussian if the one-dimensional marginals $\langle Y, z\rangle$ are sub-Gaussian for all $z\in\mathbb{R}^n$. The sub-Gaussian norm of such a vector $Y$ is defined as
$$\|Y\|_{\psi_2}=\sup_{z\in\mathbb{S}^{n-1}}\|\langle Y,z\rangle\|_{\psi_2}.$$
For any sub-Gaussian  $Y=(y_1,y_2,...,y_n)$, $\|Y\|_{\psi_2}\geq\max_i\|y_i\|_{\psi_2}$ and $\|Y_T\|_{\psi_2}\leq\|Y\|_{\psi_2}$ for any $T\subseteq \{1,\dots,n\}$.

A random vector $Y\in\mathbb{R}^n$ is called isotropic if $\mathbb{E}\langle Y,z\rangle^2=\|z\|^2$ for all $z\in\mathbb{R}^n$. For example, random vectors whose entries are independent random variables with mean zero and variance one are isotropic. Note that if $Y$ is isotropic, so is $Y_T$ for any $T\subseteq \{1,\dots,n\}$.

This paper is organized as follows: Section \ref{sec:FrameTheory} reviews basic properties of frames. Section \ref{sec:quantization_and_wnh} gives a summary of the literature regarding the use of MSQ in the frame theory setting and explains how WNH helps, albeit in a non-rigorous fashion, with the analysis of the reconstruction error. After reviewing basics of compressed sensing and quantization in Section \ref{sec:CS_and_quantization}, we present our main contributions in frame quantization in Section \ref{sec:MSQFramesReview} and the extensions of these results to the compressed sensing setting in Section \ref{sec:result_in_CS}. Preliminary facts and lemmas are stated in Section \ref{sec:preliminaries}. Proofs of all results may be found in Section \ref{sec:proofs}. Finally, we present numerical experiments Section \ref{sec:numerical_experiments}.

\section{Background and setting}
\label{sec:Background_and_setting}
Our focus is on MSQ in the settings of frame theory and compressed sensing. Next, we provide a brief description of these problems, how they are related to each other, and the state-of-the-art. 
\subsection{Frame theory}
\label{sec:FrameTheory}
In various applications, signals can be modeled as elements in a (separable) Hilbert space $H$ such as $\R^d$, $\ell^2$, or $L^2(\R^d)$. Accordingly, any signal $x$ in $H$ can be represented discretely, for example, by choosing an orthonormal basis $\{b_n\}$ for $H$ and computing (or measuring) $x_n=\langle x,b_n \rangle$. Alternatively, one can use {\em frames}, which generalize the notion of bases by allowing ``controlled'' redundancy. Specifically, $\{e_n\} \subset H$ is a frame for $H$ if there exists positive constants $A\le B$, called the frame bounds, such that for all $x\in H$,
$$ A\|x\|^2 \le \sum|\langle x,e_n\rangle |^2 \le B\|x\|^2.
$$
If $\{e_n\}$ is a frame for $H$, there exist scalars $\{y_n\}$ such that $x=\sum y_n e_n$. If $\{e_n\}$ is overcomplete, i.e., the frame is redundant, the choice of  $y_n$ in this decomposition is not unique. This can be exploited in applications to reduce the error introduced by lossy operations like transmission over a noisy channel or quantization.

One common special case is $H=\R^k$, which is our main focus in this paper. Let $m\ge k$ be positive integers and let $\{e_j\}_1^m$ be a collection of vectors in $\R^k$. Denote by $E$ the $m \times k$ matrix whose $j$th row is $e_j^T$. Then $\{e_j\}_1^m$  is a frame for $\R^k$ if and only if $E$ is full rank. In this case, all $k$ singular values of $E$ are positive and the frame bounds $A$ and $B$ are the squares of the smallest and largest singular values of $E$, respectively. Any left inverse $\widetilde{E}$ of $E$ (whose columns also form a frame for $\R^k$) is said to be a dual of $E$; the Moore-Penrose pseudo-inverse is called the \emph{canonical dual}. Note that $E$ has infinitely many duals if $m>k$, in which case we say that $E$ is redundant. Throughout, we shall use $E$ to denote both the frame and the associated analysis matrix as defined above.

\subsection{MSQ for frames and the WNH}
\label{sec:quantization_and_wnh}
Suppose that our signals of interest are in $\R^k$, and let $E\in \R^{m\times k}$ be the analysis matrix of a frame for $\R^k$. The vector of frame coefficients of a signal $x$ is given by $y=Ex$, which we interpret as \emph{measurements} of $x$. Clearly, $x$ can be recovered from $y$ exactly by means of {\it linear reconstruction methods}, i.e., $x=Fy$ where $F$ is any left-inverse of $E$. However, when we have {\it quantized} measurements $q=Q(y)$, we can recover only an approximation $\widetilde{x}$ to $x$ and we wish the reconstruction error $\|x-\widetilde{x}\|$ (in an appropriate norm) to be as small as possible.  Such approximations can be obtained using consistent reconstruction methods, which are nearly optimal in terms of accuracy, but computationally expensive \cite{Goyal,cvetkovic2003resilience,cvetkovic2001simple,powell2016error}. For frames consisting of (nearly) equal-norm vectors, the lower bound for the reconstruction error (over all reconstruction methods) is $\Omega(\lambda^{-1})$ \cite{goyal1998quantized} where $\lambda:=m/k$. 

Alternatively, one can employ linear reconstruction techniques by using an appropriate left inverse of $E$. In what follows, we restrict our attention to a common approach in the literature where one uses linear reconstruction using the Moore-Penrose pseudo-inverse $E^\dagger$ of $E$ and measures the approximation error in $\ell_2$-norm. Note that among all left-inverses of $E$, $E^\dagger$ has the smallest operator norm. Thus, if the measurement error $y-Q(y)$ has no assumed structure\footnote{It is worth to note, though, that when $y-Q(y)$ is ``shaped'', other left inverses such as ``Sobolev duals'' may yield significantly better reconstructions, e.g., \cite{BLPY,Sobolev_Duals_for_RF,NewSigmaDelta,chou2015noise,Chou2016}.} (which can be motivated using WNH, e.g., see \cite{goyal2001quantized}), $E^\dagger$ minimizes the $\ell_2$ reconstruction error.

Next, set $\widetilde{x}=E^\dagger Q(y)=E^\dagger Q(Ex)$. The corresponding reconstruction error is 
\begin{equation} \label{rec_error}
\mathcal{E}(x):=\|x-\widetilde{x}\|=\|x-E^\dagger Q(Ex)\|=\|E^\dagger(Ex-Q(Ex))\|.
\end{equation}
The aim is to  identify the dependence of $\mathcal{E}(x)$ on $\delta$, $k$, and $m$ for a given infinite family of frames parametrized by $m$ and $k$. It is often of interest to understand the worst-case error over a compact set $K$ of signals, i.e., $\sup_{x\in K} \mathcal{E}(x)$, average $L_p$ error over such a set $K$, i.e., $\|\mathcal{E}\|_{L_p(K)}$, or the expected error when the signal is drawn randomly from $K$, i.e., $\mathbb{E} \ \mathcal{E}(x)$. The analysis of these error terms is not straight-forward, especially when $m>k$. Perhaps the most common simplifying approach is to use the WNH:
 {\it If $E$ is fixed and $x$ is random, the entries of the vector $Ex-Q(Ex)$ are assumed to be i.i.d. random variables, uniformly distributed in $(-\delta/2,\delta/2]$ (so, with mean 0 and variance $\delta^2/12$)} \cite{BennettWNH}, cf. \cite{gray1990quantization,Goyal,Jimenez}. 
Indeed, if we assume that the WNH holds, it is straight-forward to show that 
\begin{equation}\label{wnh111}
\mathbb{E}\ \mathcal{E}^2(x) =\|E^\dagger\|_F^2\delta^2/12,
\end{equation}
where $\|\cdot\|_F$ denotes the Frobenius norm. Finally, recall that ${\|E^\dagger\|_F\leq\sqrt{k}\|E^\dagger\|=\sqrt{k}\left(\sigma_{\min}(E)\right)^{-1}}$ and for a vast class of matrices (e.g., matrices with independent isotropic sub-Gaussian random rows), ${\sigma_{\min}(E)\geq C\sqrt{m}}$ (with high probability). For such matrices, the WNH allows us to conclude 
\begin{equation}
\label{eq:decay_predicted_by_WNH}
\mathbb{E}\mathcal{E}^2(x)\leq C\frac{k}{m}\delta^2
\end{equation}
where $C$ is independent of $m$ and $k$. The class of matrices for which \eqref{eq:decay_predicted_by_WNH} holds includes random partial Fourier matrices \cite{RudelsonRIPForIIDRandomMatrices} and all matrices with independent isotropic sub-Gaussian random rows, e.g., \cite[p. 232]{Vershbook}, including random matrices with i.i.d. standard Gaussian or Bernoulli entries. Note that \eqref{eq:decay_predicted_by_WNH} becomes an equality if $E$ is a unit-norm tight frame; therefore, for at least such frames, the error bound is sharp under the WNH.

The WNH is rather successful for predicting the reconstruction error associated with hard-to-analyze quantization schemes \cite{powell2013quantization,Xu, gray1990quantization, Jimenez} and can be (approximately) justified in special cases  (\cite{gray1990quantization,borodachov2009lattice,Jimenez}. 
Yet, the WNH is not rigorous and, at least in certain cases, not valid---see, e.g., \cite{Jimenez,Kushner,benedetto2006sigma,Xu}. To our knowledge, there are only a few results in the literature that provide a rigorous analysis of $\mathcal{E}(x)$ that is not based on the WNH. In \cite{Xu} and \cite{Tight_fr}, the authors consider MSQ with a fixed $\delta>0$ for quantizing expansions with respect to asymptotically equidistributed unit-norm tight frames for $\R^k$. Specifically, they investigate the decay of the reconstruction error $\mathcal{E}(x)$ as more measurements are taken, i.e., as $m$ increases. An analysis based on the WNH would suggest that the error should vanish as $m$ goes to infinity. However, \cite{Tight_fr} shows that for any given signal $x\in \R^k$, if the quantizer step size $\delta$ is sufficiently small (depending on $\|x\|$ and $k$, but independent from $m$), 
$$\lim_{m\rightarrow\infty}\mathcal{E}(x)\geq C\frac{\delta^{\frac{k+1}{2}}}{\|x\|^{\frac{k-1}{2}}},$$
where $C=C(k)>0$, implying that the error does not always tend to zero as $m\rightarrow\infty$. This shows that the WNH does not hold, at least in this particular setting.

While the WNH is not rigorous and not valid for all cases, its statement can be heuristically justified based on the asymptotic behavior of $Ex-Q(Ex)$ as $\delta$ approaches $0$. For example, consider a deterministic full rank matrix $E\in\mathbb{R}^{m\times k}$ and a random signal $x$ with absolutely continuous distribution of its entries with respect to the $k$-dimensional Lebesgue measure. Then, $\frac{1}{\delta}(Ex-Q(Ex))\rightarrow[-1/2,1/2)^m$ in distribution as $\delta\rightarrow 0^+$ \cite{Jimenez,borodachov2009lattice}. Furthermore, the entries of $\frac{1}{\delta}(Ex-Q(Ex))\rightarrow[-1/2,1/2)^m$ are asymptotically uncorrelated \cite{Viswanathan}. 

Note that the results above are asymptotic as $\delta\to 0$, and therefore, they can be applied for high-resolution settings where $\delta>0$ is extremely small. However, in practice, $\delta$ is often not very small -- sometimes as big as the infinity norm of the signal $x$ (corresponding to one-bit quantization). Thus, non-asymptotic reconstruction error analysis must be performed to understand the actual behavior of the reconstruction error.

\subsection{Compressed sensing and quantization}
\label{sec:CS_and_quantization}
While our main focus is on the use of MSQ for random frame expansions, a major application is MSQ for compressed sensing (CS) measurements. We explain this after briefly reviewing relevant facts in the theory of CS and quantization. 
\subsubsection{Basics of CS}
Motivated by the empirical observation that various classes of signals encountered in practice, such as audio and images, admit (nearly) sparse approximations using an appropriate basis or frame, CS has now been established as an effective sampling theory \cite{OriginCS1,OriginCS2,OriginCS3,candes2006compressive}. Let $x$ be in $\Sigma_k^N$, i.e.,  a $k$-sparse vector in $\mathbb{R}^N$. In CS, one collects non-adaptive linear measurements of $x$, given by $y=\Phi x$. It can be shown that if $\Phi\in\mathbb{R}^{m\times N}$, with $m\ll N$, is an appropriate matrix, e.g., if it satisfies the restricted isometry property (RIP) of order $k$ with a sufficiently small constant $\delta_k$, see, e.g., \cite{OriginCS1,OriginCS2,OriginCS3}, then the solution of the $\ell_1$-minimization problem given by
$$\text{minimize} \|z\|_1\,\text{subject to }\Phi z=y$$
yields $x\in \Sigma_k^N$ exactly (or when compressible, an approximation that is nearly as accurate as the best $k$ term approximation of $x$). 

Even though RIP is a deterministic condition, checking for RIP becomes computationally intractable as the size of the matrix grows \cite{bandeira2013certifying,tillmann2014computational}. Fortunately, broad classes of random matrices are known to satisfy RIP with overwhelming probability. In particular, consider $\Phi\in\mathbb{R}^{m\times N}$ whose rows are independent isotropic sub-Gaussian random vectors (note unusual normalization where we do not normalize based on the number of rows). If $m\geq C\epsilon^{-2}k\ln(eN/k)$, where $C>0$ depends only on the sub-Gaussian
norm of the rows, $\frac{1}{\sqrt{m}}\Phi$ satisfies the RIP of order $k$ with $\delta_k<\epsilon$ for any $\epsilon \in (0,1)$ with probability at least $1-2\exp(-c\epsilon^2m)$; see, e.g., \cite{Vershbook}. Therefore, such matrices are of interest in CS.
\subsubsection{Quantization for CS}
\label{sec:Quantization_for_CS}
As in the case of frame expansions, it is crucial that the compressed measurements are quantized given that our technology is almost exclusively digital. While the focus of the early literature on CS virtually neglected the issue of quantization (and compression in the sense of source coding), recent work focused on two alternative quantization approaches. In the first approach, one considers quantization methods that ``shape'' the approximation error in a way so that it can be ``filtered out'' during reconstruction. Such methods are called {\it noise-shaping quantizers} and have recently been shown to be effective in CS quantization \cite{Sobolev_Duals_for_RF,feng2014rip,NewSigmaDelta, Chou2016,SWY2018,SWY2017}, cf. \cite{chou2015noise,Chou2016,chou2017distributed}. In fact, one can achieve exponentially accurate approximations (with respect to the total bit budget) if one incorporates a post-$\sd$-quantization coding stage, i.e., one achieves nearly optimal encoding -- see \cite{SWY2017} for details.

While noise-shaping quantizers provide superior approximations, their implementation requires ``memory''. In applications where the measurements are obtained sequentially or all at once, e.g., when acquiring an audio signal or an image, this is not an important issue. However, if the measurements are obtained, say, by a distributed sensor network, implementing noise-shaping quantizers may be challenging if the sensors cannot communicate with each other efficiently. On the other hand MSQ is easy to implement in any setting and it remains to be the most popular approach in CS quantization; see, e.g., \cite{fletcher2007rate,boufounos2015quantization,lisurvey}. Specifically, various reconstruction methods have been proposed that aim to improve the approximation obtained from the MSQ-quantized compressive measurements \cite{zymnis2010compressed,MJCDV,Jacques2013,jacques2011dequantizing}. 
In this paper, we also focus on the use of {\it multi-bit} MSQ in the CS framework.\footnote{Note that one-bit MSQ provides a yet simpler quantization method. The analysis of the one-bit MSQ turns out to be significantly different from the multi-bit MSQ and is not within the scope of this paper, e.g., see \cite{boufounos2010reconstruction,plan2013one,ai2014one}.} That is, we will quantize the compressed measurements $y$, as defined above, by $Q(y)$, where $Q=Q_\delta^{\rm{MSQ}}$. One may interpret the associated quantization error $y-Q(y)$ as ``noise'' that is bounded by $\delta/2$ in $\ell_\infty$-norm (and consequently, by $\sqrt{m}\delta/2$ in $\ell_2$-norm). Such an approach allows us to apply classical robust recovery results in CS. Specifically, $\widehat{x}$ given by
\begin{equation}\label{coarse11}
\widehat{x}:=\arg\min \|z\|_1 \text{ subject to } \|\Phi z-Q(y)\|\leq \sqrt{m}\delta/2
\end{equation}
satisfies 
\begin{equation}\label{err1}
\|x-\widehat{x}\|\leq C\delta
\end{equation}
with high probability when $\lambda=\frac{m}{k}$ is sufficiently large and $\Phi$ is appropriately normalized as $m$ increases (which is crucial to ensure that the constant $C$ in \eqref{err1} does not depend on $m$ -- see \cite{Robust_cs} and \cite{Sobolev_Duals_for_RF}). Note, that while the approximation error in \eqref{err1}  scales linearly with the step size $\delta$ (as expected), it does not tend to zero as the ``redundancy'' $\lambda$ increases; this is observed also in numerical experiments in \cite{Sobolev_Duals_for_RF}. 

While the above mentioned approach follows naturally from the basic robust recovery results in CS, the reconstruction can be improved by employing alternative techniques that provide more accurate estimates, see, e.g., \cite{jacques2011dequantizing,boufounos2015quantization,Sobolev_Duals_for_RF}
and also \cite{boufounos20081,jacques2013robust,plan2013one} for the 1-bit case. One such alternative is the use of the ``two-stage framework'' as devised in \cite{Sobolev_Duals_for_RF}, which can be used to adopt any given frame quantization technique to the CS setting to improve the reconstruction. Next we describe this framework. 

Let $\Phi$ be an $m\times N$ CS measurement matrix, and let $y=\Phi x$ be the compressive samples of $x\in \Sigma_k^N$. Suppose that $T$ is the support of $x$, i.e., $T:=\{j: \ x(j)\ne 0\}$ with $\#T=k$. Then, $\Phi_T$ is the analysis matrix of a frame for $\R^k$ that consists of $m$ vectors and  the compressive samples satisfy $y=\Phi_T x_T$, i.e., the entries of $y$ are the frame coefficients of $x_T$ with respect to $\Phi_T=:E$. Note that if, for example, $\Phi$ is a random matrix with i.i.d. sub-Gaussian entries, then $\Phi_T$ is a \emph{random frame}  with i.i.d. sub-Gaussian entries. This observation motivates the use of any frame quantization method, in our case MSQ, along with a two-stage reconstruction scheme, as proposed in \cite{Sobolev_Duals_for_RF}.  
\subsection*{Two-stage reconstruction scheme for MSQ in CS}
Let $q=Q(y)$ with $Q=Q_\delta^{\rm{MSQ}}$. In {\bf Stage 1}, we recover a coarse approximation $x^\#_{\rm MSQ}$ by solving \eqref{coarse11}, 
from which the support $T$ of $x$ can be extracted with high probability under some additional conditions on $x$ -- see \cite{Sobolev_Duals_for_RF}. In {\bf Stage 2}, we obtain a refined estimate $x_{\rm MSQ}$ of $x_T$ by using a reconstruction method tailored to the underlying frame quantization problem. Our focus in Stage 2 will be the use of {\it linear reconstruction methods}. Specifically,  the refined estimate will be obtained via $x_{\rm MSQ}=\Phi_T^\dagger q$. Our goal is to analyze the approximation error $\|x-x_{\rm MSQ}\|_2$ as $\lambda$ increases.  

The error analysis for this two-stage scheme raises the following question in random frame theory: For an $m\times k$ random matrix $E$ with $m>k$ and for a deterministic signal $x\in\mathbb{R}^k$, how does $\|x-E^\dagger Q(Ex)\|$ decay as $\lambda$ grows? We focus on random matrices with independent, isotropic, sub-Gaussian rows. While the technique of \cite{Sobolev_Duals_for_RF} can be generalized to this case, it yields only a (non-zero) constant error bound independent of  $\lambda$. On the other hand, the WNH ``predicts'' an approximation with $O(\lambda^{-1/2})\delta$ error, which appears to agree with numerical experiments as seen, for example, in Figure \ref{fig:MSQforCS}.  Here, we tighten both the result in  \cite{Sobolev_Duals_for_RF} as well as the prediction using the WNH. Specifically, we prove that the reconstruction error behaves like $v+O(\lambda^{-1/2})\delta$, where $v>0$ is a small number that is barely noticeable in applications; this constant term, however, can not be removed from the error bound. We also extend this result to the CS setting when the recovery is performed using the two-stage algorithm stated above.

\section{MSQ for random frames}
\label{sec:MSQFramesReview}

Let $m\ge k$ and let $E \in \R^{m \times k}$ be a sub-Gaussian random matrix with independent isotropic rows. Suppose that $x\in\mathbb{R}^k$ is deterministic and fixed, and $Q=Q^{\text{MSQ}}_\delta$. We wish to control  the reconstruction error $\mathcal{E}(x)=\|x-E^\dagger Q(Ex)\|$ by establishing upper and lower bounds. First, since $\|Ex-Q(Ex)\|_\infty\le \delta/2$ and $\sigma_{\min}(E)\ge c\sqrt{m}$ with overwhelming probability  { \cite[p. 23 and p. 36]{Vershbook}}, a rough product bound yields
\begin{equation}\label{eq:rough_bound_on_reconstruction error}
\|x-E^\dagger Q(Ex)\|\leq \left(\sigma_{\min}(E)\right)^{-1}\sqrt{m}\delta/2 \le C \delta
\end{equation}
with overwhelming probability. 

Note that this bound captures the dependence of the error on the quantizer resolution $\delta$; but it does not depend on $m$, the number of measurements. Intuitively, we expect the error to decrease as we obtain more measurements, i.e., more information about $x$. This motivates further analysis of the approximation error $\mathcal{E}(x)$.

\noindent{\bf A heuristic bound.} Before we refine this bound rigorously in Section \ref{sec:mainresults}, we present a heuristic estimate based on a modified version of the WNH that shows that approximation error decays as $m$ increases. As we showed in Section \ref{sec:quantization_and_wnh}, such a decay can be ``justified'' using the WNH when the frame $E$ is deterministic while the signal is random. In our current setting the frame is random, so the WNH is not directly applicable: On one hand, since $x$ is deterministic and the random matrix $E$ has independent rows, the entries of $u:=Ex-Q(Ex)$ are independent random variables; if we additionally assume that the rows of $E$ are identically distributed, the entries of $u$ are identically distributed as well. This observation coincides in part with the WNH. On the other hand, \eqref{wnh111} does not hold anymore since $E^\dagger$ is random in our setting. One way to go around this is to introduce a modified WNH as follows. 

\medskip

\noindent {\bf Modified WNH (m-WNH):} \textit{Let $E$ be a random matrix. Assume that the signal $x$ is also random and independent of $E$. The m-WNH assumes that the conditional random variable $\left(Ex-Q(Ex)\right)|E$ has i.i.d. entries that are uniformly distributed in $(-\delta/2,\delta/2]$.}
\medskip

\noindent{\bf Implications of m-WNH:}
Set $E^\dagger=(e^\dagger_{ij})$, $u:=\left(Ex-Q(Ex)\right)$, and suppose that m-WNH holds. Then
\begin{align*}\mathbb{E}\left\{\|x-E^\dagger Q(Ex)\|^2|E\right\}&=\mathbb{E}\left\{\|E^\dagger\cdot (Ex-Q(Ex))\|^2|E\right\}\\
&=\sum_{i=1}^k\sum_{j=1}^m\sum_{s=1}^m e^\dagger_{ij}e^\dagger_{is}\mathbb{E}\left\{u_ju_s|E\right\}=\sum_{i=1}^k\sum_{j=1}^m\sum_{s=1}^m e^\dagger_{ij}e^\dagger_{is}\mathbb{E}\left\{u^2_j|E\right\}\one_{[j=s]}\\
&=\sum_{i=1}^k \sum_{j=1}^m(e^\dagger_{ij})^2 \frac{\delta^2}{12}=\frac{\delta^2}{12}\|E^\dagger\|_F^2\\
&\leq \frac{k \delta^2}{12}\|E^\dagger\|^2=\frac{k \delta^2}{12}\left(\sigma_{\min}(E)\right)^{-2}.\\
\end{align*}
Finally, using the law of total expectation,
\begin{align}
\mathbb{E}\|x-E^\dagger Q(Ex)\|^2&=\mathbb{E}\left(\mathbb{E}\left\{\|x-E^\dagger Q(Ex)\|^2|E\right\}\right) \notag \\
&\leq\mathbb{E}\frac{k \delta^2}{12}\left(\sigma_{\min}(E)\right)^{-2} \notag \\
&\leq \frac{k \delta^2}{12}\cdot Cm^{-1} \label{mWNH_bound}
\end{align}
In the last inequality, we used that $\sigma_{\min}(E)\ge c\sqrt{m}$ with overwhelming probability. In this case,
$\left(\mathbb{E}\{\mathcal{E}^2(x)\}\right)^{1/2} \leq C\delta\sqrt{k/m}$
provided that both $x$ and $E$ are random and the m-WNH holds.

Numerical experiments in Section \ref{sec:numerical_experiments} appear to agree with the heuristic calculation above: the empirical reconstruction error is $O(\lambda^{-1/2})\delta$, as predicted in \eqref{mWNH_bound} when $E$ is drawn from various random matrix ensembles. On the other hand, WNH is a special case of m-WNH, thus m-WNH is also not fully rigorous and not valid at least in certain cases. For example, in Corollary \ref{thm:MSQForFrame}, we show that in the case when the matrix $E$ has i.i.d. standard Gaussian random entries (we often say such an $E$ is a {\it Gaussian random matrix}), m-WNH does not hold. This motivates our error analysis in the rest of the paper that does not rely on m-WNH.

\subsection{Error estimates without WNH -- main results}
\label{sec:mainresults}

From here on, let $E\in\mathbb{R}^{m\times k}$ with $m>k$, $\delta>0$, $Q:=Q_\delta^{\rm MSQ}$, and $\lambda:=m/k$. We seek to estimate the reconstruction error $\mathcal{E}(x)=\|x-E^\dagger Q(Ex)\|$ for $x\in \R^k$.

\begin{theorem}
\label{thm:MSQforFrameGeneral}
Let $k\geq 3$ and let $x\in\mathbb{R}^k$ be fixed. Suppose that $\lambda>1$ and $E\in\mathbb{R}^{m\times k}$ is a random matrix with independent isotropic sub-Gaussian rows whose $\psi_2$-norm does not exceed $K$. Then, there is an absolute constant $C>0$ such that for every $c_2\in(0,1)$, setting  $c_1=CK\sqrt{\ln (e^2c_2^{-1})}>0$, we have
\begin{equation}
\label{eq:main_for_frames}
\mathcal{E}(x)<A\left(\mu+c_1\sqrt{\log k}\lambda^{-1/2}\delta\right)
\end{equation}
and
\begin{equation}
\label{eq:lowerboundthm}
\mathcal{E}(x)>A'\left(\mu-c_1\sqrt{\log k}\lambda^{-1/2}\delta\right),
\end{equation}
with probability at least $1-c_2-2\exp(-c_3m)$. Here $\mu=\frac{1}{m}\|\mathbb{E}E^T(Ex-Q(Ex))\|$,  $A=(1/2-c_K\lambda^{-1/2})^{-2}>0$, $A'=(3/2+c_K\lambda^{-1/2})^{-2}>0$, $c_K>0$ depends only on $K$, and $c_3>0$ is an absolute constant.
\end{theorem}

\begin{remark}
The probability of failure contains a constant term $c_2$. This constant is unavoidable, but it may be chosen to depend on $m$ and/or $\lambda$. Note that such a change will affect $c_1$.
\end{remark}

The error bounds in (\ref{eq:main_for_frames}) are composed of (constant multiples of) two summands: $\mu=\frac{1}{m}\|\mathbb{E}E^T(Ex-Q(Ex))\|$ and $\sqrt{\log k}\lambda^{-1/2}\delta$. As the decay rate of the latter agrees with that predicted by the m-WNH, we will focus on the first term, i.e., $\mu$, for random matrices $E$ with identically distributed rows. 

\begin{proposition}
	\label{prop:constant_term_is_constant}
Assume that the conditions of Theorem \ref{thm:MSQforFrameGeneral} hold and, in addition, the rows $e_i^T$ of $E$ are identically distributed.  Then,
$$\mathbb{E}E^T(Ex-Q(Ex)) =  m\mathbb{E} e_1(e_1^Tx-Q(e_1^Tx)).$$
This implies
$$\mu=\|\mathbb{E} e_1(e_1^Tx-Q(e_1^Tx))\|.$$
\end{proposition}
\begin{proof}
Fix $x\in\R^k$. Then,
$$\mathbb{E}E^T(Ex-Q(Ex))=\sum_{i=1}^m \mathbb{E} e_i(e_i^Tx-Q(e_i^Tx)),$$
where the summands on the right hand side are identical vectors all equal to, say, $\mathbb{E} e_1(e_1^Tx-Q(e_1^Tx))$. It follows that
$$\mu=\frac{1}{m}\|\mathbb{E}E^T(Ex-Q(Ex))\|=\|\mathbb{E} e_1(e_1^Tx-Q(e_1^Tx))\|$$ which completes the proof. 
\end{proof}
\begin{remark}
Proposition \ref{prop:constant_term_is_constant} shows that, if matrix $E$ has i.i.d. isotropic sub-Gaussian rows, then $\mu$ in Theorem \ref{thm:MSQforFrameGeneral} is a constant that does not depend on $m$. Thus, in cases when $\mu\ne 0$, e.g., when $E$ is a Gaussian matrix -- see Corollary \ref{thm:MSQForFrame}, $\mathcal{E}(x)=O(1)$, rather than $O(m^{-1})$ which is what the m-WNH predicts. 

\end{remark}

Next, we establish an upper bound for $\mu$.
\begin{proposition} \label{prop:bound1}
In the setting of Theorem \ref{thm:MSQforFrameGeneral}, 
$
\mu\leq \frac{\delta}{2}(1+c_K\sqrt{\frac{k}{m}}).
$
\end{proposition}
\begin{proof}
\begin{align*}
\mu=\frac{1}{m}\|\mathbb{E}E^T(Ex-Q(Ex))\|&\leq
\frac{1}{m}\mathbb{E}\left(\sigma_{\max}(E^T) \|Ex-Q(Ex)\|\right)\\
&\leq \frac{1}{m}\mathbb{E}\left(\sigma_{\max}(E^T)\frac{\delta}{2}\sqrt{m}\right)\\
&\leq \frac{\delta}{2}(1+c_K\sqrt{\frac{k}{m}}).
\end{align*}
For the bound on $\mathbb{E}\sigma_{\max}(E^T)$, we used Theorem \ref{thm:matrixnormestGeneral}.
\end{proof}
The above estimate implies that, in the worst-case scenario, $\mu$ is bounded by $\delta/2+\epsilon$ as $m\to \infty$.

Indeed, the bound of order $\delta/2$ is observed {\it numerically} if the random frame $E$ is a submatrix of the first $k$ columns of a sufficiently large Fourier matrix with $m$ randomly selected rows, see Figure \ref{fig:MSQ_Fourier}.

The following artificial scenarios illustrate (provably) that the reconstruction error indeed does not always decay to zero but may tend to a non-zero constant, possibly (nearly) as big as the upper bound in Proposition \ref{prop:constant_term_is_constant}. 
\begin{enumerate}[(A)]
		\item  		\label{case:artif_example_Bernoulli1D}
		{[{\it Bernoulli frame, one-dimensional signal]}} Let $k=1$, i.e., $x \in \R$, and let $E\in\mathbb{R}^{m\times 1}$ be a $\pm1$ Bernoulli frame. Then each individual sample is either $x$ or $-x$, i.e., more samples do not bring any additional information. Thus, the optimal reconstruction from quantized measurements $Q(Ex)$ is $\widetilde{x}=Q(x)$. Accordingly, if $Q(x)\neq x$, which is the case for almost all $x$, the reconstruction error $|x-\widetilde{x}|$ is non-zero, and in fact, it can be as big as $\delta/2$. Note that the same discussion applies if we consider a Bernoulli random frame with $k>1$ and set $x=c (1,0,0,...,0)\in \R^k$.
		\item $ $[{\it Bernoulli frame, a class of high-dimensional signals}]
		Let the frame $E\in\mathbb{R}^{m\times k}$ be a $\pm1$ Bernoulli frame. Consider all $x=(x_1,x_2,...,x_k)$ such that 
		\begin{equation}
		\label{eq:artif_examples_Bernoulli}
		|x_i-Q(x_i)|<\frac{\delta}{2k}
		\end{equation} 
		Using that all entries of $E=(e_{ij})$ are $\pm1$, we get that for all $i\in\{1,2,..,m\}$, $$ -\delta/2+(EQ(x))_i < (Ex)_i < \delta/2+(EQ(x))_i $$
This, together with the fact that $(EQ(x))_i \in \delta\Z$, implies that $Q(Ex)=EQ(x)$ for all $x$ satisfying \eqref{eq:artif_examples_Bernoulli}. Accordingly, for two signals $x^1$ and $x^2$ such that $Q(x^1)=Q(x^2)$ and \eqref{eq:artif_examples_Bernoulli} holds, $Q(Ex^1)$ equals $Q(Ex^2)$, and thus the reconstruction will be identical, yielding an $\ell_2$ reconstruction error as big as $\frac{\delta}{2\sqrt{k}}$.
		
		\item $ $[{\it Bernoulli random frame, $m>2^k$}]
                  Consider a $\pm1$ Bernoulli matrix $E\in\mathbb{R}^{m\times k}$ and any signal $x\in\mathbb{R}^k$ where $m>2^k$. Note that such a Bernoulli matrix can have at most $2^k$ distinct rows, i.e., if we exclude repetitive rows, we get at most $2^k$ different measurements. Once the frame $E$ contains all $2^k$ distinct rows, more measurements do not bring any additional information. Therefore, unless $Ex=Q(Ex)$, the reconstruction error will saturate at a non-zero constant.
		
		\item $ $[{\it Fourier random frame, a class of high-dimensional signals}] Consider the discrete Fourier transform (DFT) matrix $F\in\mathbb{R}^{N\times N}$. We select the first $d$ columns of $F$ and we draw rows uniformly at random. Suppose that the signal $x=(c,0,0,...,0).$ Then, all samples of the signal will be identical, and similar to the case \ref{case:artif_example_Bernoulli1D}, the reconstruction error for some such signals is very close to $\delta/2$.  \end{enumerate}

While some of these scenarios attain (nearly) the bound of Proposition \ref{prop:bound1}, this bound is not tight, at least for some ``nice'' ensembles. In Corollary \ref{thm:MSQForFrame}, we show that $\mu$ is not zero when $E$ is Gaussian, but $\mu$ it is very small in any practically relevant setting, which we prove in the next section.

\subsubsection{Gaussian random matrices.}
\label{sec:Gaussian_matrices}

Next, we obtain sharp bounds for $\|\mathbb{E}E^T(Ex-Q(Ex))\|$ when $E$ is a Gaussian random matrix. These, in turn, yield bounds for the reconstruction error $\mathcal{E}(x)=\|x-E^\dagger Q(Ex)\|$.

\begin{theorem}
\label{thm:estmean}
Let $x\in\mathbb{R}^k$ be fixed and let  $E\in\mathbb{R}^{m\times k}$ be a random matrix with isotropic sub-Gaussian rows. Furthermore, assume that the entries of $E$ are i.i.d. whose density function $\phi$ is a Schwartz function. Then the $i$th entry of $\mathbb{E}E^T(Ex-Q(Ex))$ satisfies
\begin{equation}
\label{eq:OneEntry_Schwartz}
\left(\mathbb{E}E^T(Ex-Q(Ex))\right)_i=m\left(x_i+x_i\sum_{p\in\mathbb{Z}}\left((-1)^p\widehat{g}(\frac{|x_i|}{\delta}p)\prod_{s\neq i}\widehat{\phi}\left(\frac{x_s\text{sign}(x_i)}{\delta}p\right)\right)\right),
\end{equation}
where $g(z):=\mathbb{E}e_{11}\one_{\{e_{11}\leq z\}}=\int_{-\infty}^z t\phi(t)dt.$

In particular, if $E$ is a Gaussian matrix, 
$$\left(\mathbb{E}E^T(Ex-Q(Ex))\right)_i=-2mx_i\sum_{p=1}^\infty(-1)^p\exp(-\frac{2\pi^2\|x\|^2p^2}{\delta^2}),$$
which implies	
\begin{equation}
\label{eq:single_entry_Gaussian_est}	
2\|x\|\left(\exp(-\frac{2\pi^2\|x\|^2}{\delta^2})-\exp(-\frac{8\pi^2\|x\|^2}{\delta^2})\right)<\mu< 2\|x\|\exp(-\frac{2\pi^2\|x\|^2}{\delta^2}),
\end{equation}
where $\mu=\frac{1}{m}\| \mathbb{E}E^T(Ex-Q(Ex))\|$ as in Theorem \ref{thm:MSQforFrameGeneral}.
\end{theorem}

\begin{remark}
The right-hand side of \eqref{eq:OneEntry_Schwartz} may be to calculate exactly. On the other hand, since $\phi$ is a Schwartz function (and so are $\widehat{\phi}$ and $\widehat{g}$), it may be accurately approximated by truncating the series to its first few terms.
\end{remark}

\begin{remark} For Gaussian random matrices, the term $\mu$ is not larger than $2\|x\|\exp(-\frac{2\pi^2\|x\|^2}{\delta^2})$, which is typically very small. For example, if $\|x\|=1$ and $\delta=0.5$, this bound is approximately $2.05\times 10^{-34}$.
\end{remark}

We end this section by combining Theorem \ref{thm:MSQforFrameGeneral} and Theorem \ref{thm:estmean} to state the lower and upper bounds on the approximation error for Gaussian frames. 

\begin{corollary}[Reconstruction error for Gaussian frames]
\label{thm:MSQForFrame}
Let $k\geq 3$, $x\in\mathbb{R}^k$, and suppose that $E$ is an $m\times k$ Gaussian matrix (i.e., its entries are i.i.d. standard Gaussian). Then, there is an absolute constant $C>0$ such that for every $c_2\in(0,1)$, setting  $c_1=C\sqrt{\ln (e^2c_2^{-1})}>0$, we have
\begin{equation}
\mathcal{E}(x)<A\left(2\|x\| e^{-\frac{2\pi^2\|x\|^2}{\delta^2}}+c_1\sqrt{\log k}\lambda^{-1/2}\delta\right)
\end{equation}
and
\begin{equation}
\mathcal{E}(x)>A'\left(2\|x\| e^{-\frac{2\pi^2\|x\|^2}{\delta^2}}-2\|x\|e^{\frac{-8\pi^2\|x\|^2}{\delta^2}}+c_1\sqrt{\log k}\lambda^{-1/2}\delta\right)
\end{equation}
with probability at least $1-2e^{-m/8}-c_2$ for every $m\in\mathbb{N}$ such that $\lambda> 4$.
Here $A=(1/2-\lambda^{-1/2})^{-2}$ and $A'=(3/2+\lambda^{-1/2})^{-2}$.
\end{corollary}

\subsection{On the value of $\mu$}
\label{sec:on_the_value_of_mu:subsection}
In Section \ref{sec:mainresults}, we established that the reconstruction error is of order $\Omega(\mu + \frac{k}{m}\delta)$, where $\mu$ does not depend on $m$. Therefore, non-asymptotic decay of the reconstruction error is dramatically different when $\mu=0$ and when $\mu\neq 0$. Recall that Corollary \ref{thm:estmean} provides a sharp estimate of the value of $\mu$ for Gaussian random matrices, and, in particular, $\mu\neq 0$. In this section, we show that for a large $k$ such that $m/k > 3$ and for $E\in\R^{m\times k}$ with centered i.i.d. sub-Gaussian random variables, $\mu\neq 0$ under weak assumptions on $x$.

Let us start from a heuristic explanation of why one would expect $\mu\neq 0$ for a large $k$. Assume that entries of $E$ are centered i.i.d. sub-Gaussian random variables whose variance is one and whose sub-Gaussian norm does not exceed $K$. Using algebraic manipulations and the Cauchy-Schwarz inequality, we conclude
$$\mu = \frac{1}{m}\|\mathbb{E}E^T F(Ex)\|=\|\mathbb{E}e_1F(e_1^T x)\|\geq\frac{\left|\langle \mathbb{E}e_1F(e_1^T x)\,,\, x\rangle\right|}{\|x\|}=\frac{\left|\mathbb{E}e_1^TxF(e_1^Tx)\right|}{\|x\|}.$$
Therefore, if we show that $\mathbb{E}e_1^TxF(e_1^Tx)\neq 0$, it would imply that $\mu\neq 0.$ Let $z=e_1^Tx$ and consider $\mathbb{E}zF(z)$. Note that for $x=\frac{1}{\sqrt{k}}(1,1,...,1)^T$,
$$z=e_1^T x = \sum_{s=1}^k e_{1s}x_s=\frac{\sum_{s=1}^k e_{1s}}{\sqrt{k}}.$$
When $k\rightarrow\infty$, by the Central Limit Theorem, $z\xrightarrow[k\rightarrow\infty]{d} \mathcal{N}(0,1)$ for many distributions of $e_1$. Moreover, Theorem \ref{thm:estmean} for $x=(1,0,0,...,0)$ implies that for standard Gaussian random variable $\xi$,
$$\mathbb{E} \xi F(\xi) = -2\sum_{p=1}^\infty (-1)^pe^{-\frac{2\pi^2 p^2}{\delta^2}}$$
which is not zero. Combining all of the above, we conclude that for $x=\frac{1}{\sqrt{k}}(1,1,...,1)^T$,
$$\mu \geq \left|\mathbb{E}e_1^TxF(e_1^Tx)\right| \rightarrow \left|\mathbb{E} \xi F(\xi)\right|\neq 0.$$

The rigorous non-asymptotic bound based on heuristics above is provided in the following theorem.
\begin{theorem}
\label{thm:value_of_mu_bound:theorem}
Let $k$ and $m$ be such that $m/k>3$. Let $x\in\mathbb{R}^k$ be fixed and unit-norm. Suppose that $A\in\mathbb{R}^{m\times k}$ is a matrix whose entries are centered i.i.d. sub-Gaussian random variables whose variance is one and whose sub-Gaussian norm does not exceed $K$. Then, there exist constants $c$ and $c_0$ independent of $m$, $k$, and $\delta$ (defined in the Hoeffding inequality and in the Berry-Esseen inequalities accordingly) and such that for every $C\in\mathbb{N}$,
$$\mu \geq 2\left(e^{-\frac{2\pi^2}{\delta^2}} - e^{-\frac{8\pi^2}{\delta^2}}\right) - \frac{1}{3}c_0\|x\|_3^3 C(2C+1)(2C+5)\delta^2 - e\frac{K^2}{c}\exp\left(-\frac{cC^2\delta^2}{K^2}\right) - 4\exp\left(-\frac{C^2\delta^2}{2}\right).$$
\end{theorem}
\begin{remark}
It is ambiguous how $x$ changes when $k\rightarrow\infty$. If we add zero entries as $k$ grows, we would expect the same decay rate as for the lower dimensional vectors. Therefore, the entries of $x$ must change as $k$ grows.

As an extreme case, if $x=\frac{1}{\sqrt{k}}(1,1,...,1)$, then $\|x\|^3_3 = \frac{1}{\sqrt{k}}$. The norm inequality $\|x\|_2\leq k^{\frac{1}{2}-\frac{1}{3}}\|x\|_3$ implies that $O(\frac{1}{\sqrt{k}})$ is the fastest possible decay of $\|x\|_3^3$. However,
$$\|x\|_3^3=\sum_{i=1}^k |x_i|^3\leq \sum_{i=1}^k\left(|x_i|^2\max_{1\leq i\leq k}|x_i|\right)\leq \|x\|_\infty \|x\|_2^2$$
implies that for the unit-norm $x\in\mathbb{R}^k$, if $\lim_{k\rightarrow\infty}\|x\|_\infty=0$, then $\|x\|^3_3\rightarrow 0$ as $k\rightarrow 0.$
\end{remark}
\begin{remark}
Note that the number of measurements $m$ is not part of the bound on $\mu$ in the theorem. Suppose that $\|x\|_3^3\rightarrow 0$ as $k\rightarrow\infty$. Then, for large enough $k$, $\mu>0$ regardless of the number of measurements. It implies that for large enough $k$ and any number of measurements, the reconstruction error is bounded from below by a constant term and, therefore, does not diminish to zero.
\end{remark}

\subsection{Extension to other algorithms}
We extend the results above for dithered quantization and MSQ for noisy measurements.

\subsubsection{Dithered quantization}
Let $\tau=(\tau_1,\tau_2,...,\tau_m)$ be a random vector whose entries are i.i.d. uniform over $(-\delta/2,\delta/2]$ and is independent of $E$. We call this vector a dither, and we say that the quantization is dithered if the dither is added just before the quantization. In other words, the quantized measurements become 
$$q = Q(Ex+\tau).$$
One of the benefits of applying dither is that $E$ and $Ex+\tau - Q(Ex+\tau)$ are independent, and, therefore, the WNH holds. Recall that the WNH implies that the reconstruction error is $O(\lambda^{-1/2})\delta$, where $\lambda = m/k$.

Suppose that the reconstruction from quantized dithered measurements is implemented via the Moore-Penrose pseudoinverse matrix. Then,
$$\|x - x^{\text{dither}}\| =\|x - E^\dagger Q(Ex +\tau)\|= \|(E^TE)^{-1}E^T(Ex - Q(Ex+\tau))\|.$$
\begin{theorem}
\label{thm:dithered_quantization:theorem}
Let $k\geq 3$ and let $x\in\mathbb{R}^k$ be fixed. Suppose that $\lambda>1$ and $E\in\mathbb{R}^{m\times k}$ is a random matrix with independent isotropic sub-Gaussian rows whose $\psi_2$-norm does not exceed $K$. Then, there is an absolute constant $C>0$ such that for every $c_2\in(0,1)$, setting $c_1=CK\sqrt{\ln (e^2c_2^{-1})}>0$, we have
\begin{equation}
\label{eq:main_for_dither}
\|x-x^{\text{dither}}\|<A\left(2c_1\sqrt{\log k}\lambda^{-1/2}\delta\right)
\end{equation}
with probability of at least $1-c_2-2\exp(-c_3m)$.
\end{theorem}
\begin{remark}
This theorem is well-known in the dithering literature. We present it here because this result is a natural extension of the methodology of this paper.
\end{remark}

\subsubsection{Noisy measurements}
In practice, the quantized measurements $Q(Ex)$ can be distorted by noise before being stored/processed. Assume that the noise vector $\epsilon=(\epsilon_1,\epsilon_2,...\epsilon_m)^T$ is added after the quantization process, and now we want to recover $x\in\mathbb{R}^k$ from
$$\widetilde{q} = Q(Ex)+\epsilon.$$
If the signal is recovered by the Moore-Penrose pseudoinverse, the reconstruction error becomes
$$x -\widetilde{x} = x - E^\dagger \widetilde{q}=(E^TE)^{-1}E^T(Ex - Q(Ex)-\epsilon).$$
Therefore,
$$\|x-\widetilde{x}\|\leq \| (E^TE)^{-1}E^T(Ex - Q(Ex))\|+\|E^\dagger\epsilon\|.$$
The first term is estimated earlier in the paper (see Theorem \ref{thm:MSQforFrameGeneral}). 

For the second term, we provide two error bounds. Proceeding in a straight-forward way,
$$\|E^\dagger\epsilon\|\leq \frac{1}{\sigma_{\min}(E)}\|\epsilon\|\leq c\sqrt{\frac{k}{m}}\|\epsilon\|\quad\text{whp.}$$
Here the second inequality holds with overwhelming probability. Note that the above estimate is valid for any noise vector, including adversarial (i.e., worst-case) and those that are dependent on the measurement matrix.

If we also assume that the noise is independent of $E$, centered, and its entries are i.i.d. sub-Gaussian random variables with the $\psi_2$-norm at most $L$, the error bound can be improved. As before,
$$\|E^\dagger\epsilon\|\leq \frac{1}{\sigma_{\min}^2(E)}\|E^T\epsilon\|=\frac{1}{\sigma_{\min}^2(E)}\|\sum_{i=1}^m e_i \epsilon_i\|.$$
Again, $(e_i, \epsilon_i)$ are i.i.d. Therefore, we consider the norm of the sum of i.i.d. random vectors $e_i\epsilon_i.$ Note that
$$\mathbb{E} e_1\epsilon_1 = \mathbb{E}e_1\mathbb{E}\epsilon_1=0,$$
where we used that $e_i$ and $\epsilon_i$ are independent and $\mathbb{E}\epsilon_i = 0$: 
$$\mathbb{E} e_i\epsilon_i = \mathbb{E}e_i\mathbb{E}\epsilon_i =0.$$

In order to apply the Hoeffding inequality (Theorem \ref{thm:Hoeffding}), we show that the sub-Gaussian norm of $e_1\epsilon_1$ is at most $KL$. Indeed, for every $x\in\mathbb{R}^{k\times 1}$ and for every $|t|\leq\frac{1}{KL}$,
\begin{align*}
\mathbb{E}e^{t\langle e_1\epsilon_1\,,\, x\rangle }&=\mathbb{E}e^{t\langle e_1\, x\rangle \epsilon_1} =\mathbb{E}_{e_1}\mathbb{E}_{\epsilon_1} e^{\left(t\langle e_1\, x\rangle\right) \epsilon_1}\\
&\leq \mathbb{E}_{e_1} e^{t^2(\langle e_1\, x\rangle)^2L^2}\\
&\leq e^{t^2 L^2 K^2}\\
\end{align*}

Applying the Hoeffding inequality (Theorem \ref{thm:Hoeffding}) implies the following theorem.
\begin{theorem}
\label{thm:noisy_measurements}
Let $k\geq 3$ and let $x\in\mathbb{R}^k$ be fixed. Suppose that $\lambda>1$ and $E\in\mathbb{R}^{m\times k}$ is a random matrix with independent isotropic sub-Gaussian rows whose $\psi_2$-norm does not exceed $K$. Assume that noise $\tau=(\tau_1,\tau_2,...,\tau_m)$ is independent of $E$. In addition, assume that all $\tau_i$ are i.i.d. centered sub-Gaussian random variables with $\psi_2$-norm at most $L$. Then, there is an absolute constant $C>0$ such that for every $c_2\in(0,1)$, setting  $c_1=CKL\sqrt{\ln (e^2c_2^{-1})}>0$, we have
\begin{equation}
\label{eq:main_for_frames:noisy}
\mathcal{E}(x)<A\left(\mu+c_1\sqrt{\log k}\lambda^{-1/2}\delta\right)
\end{equation}
and
\begin{equation}
\label{eq:lowerboundthm:noisy}
\mathcal{E}(x)>A'\left(\mu-c_1\sqrt{\log k}\lambda^{-1/2}\delta\right),
\end{equation}
with probability at least $1-c_2-2\exp(-c_3m)$. Here $\mu=\frac{1}{m}\|\mathbb{E}E^T(Ex-Q(Ex))\|$,  $A=(1/2-c_K\lambda^{-1/2})^{-2}>0$, $A'=(3/2+c_K\lambda^{-1/2})^{-2}>0$, $c_K>0$ depends only on $K$, and $c_3>0$ is an absolute constant.
\end{theorem} 

\section{An important extension: MSQ for compressed sensing}
\label{sec:result_in_CS}
\subsection{Two-phase reconstruction algorithm}
To generalize Theorem \ref{thm:MSQforFrameGeneral} to the CS setting, we follow the procedure that is described in detail in \cite{Sobolev_Duals_for_RF}: Suppose $x\in\Sigma_k^N$ is the signal to be acquired and $\Phi\in\mathbb{R}^{m\times N}$, $m\ll N$, is a CS measurement matrix whose rows are independent isotropic sub-Gaussian random vectors with sub-Gaussian norm at most $K$. The goal is to recover a $k$-sparse signal $x$ from its quantized compressive measurements $Q(y)=Q(\Phi x)$.

We adopt the two-stage approach as summarized in Section \ref{sec:Quantization_for_CS}. Recall that in Stage 1 we recover the support of $x$ by using \eqref{eq:ell_1}. Then
the coarse approximation $\widehat{x}_{\rm MSQ}$ obtained as in \eqref{coarse11}, which satisfies $\|x-\widehat{x}_{\rm MSQ}\|\le C \delta$, where $C$ is independent of $m$. Thus, setting $k'=k$, $\eta=C\delta$, and $x'=\widehat{x}_{\rm MSQ}$ in \cite[Proposition 4.1]{Sobolev_Duals_for_RF}, we observe that the support of $x$ is fully recovered from the indexes of the largest (in-magnitude) entries of $\widetilde{x}_{\rm MSQ}$ if $|x_j|>\sqrt{2}C \delta$ for all $j$ in the support of $x$. 

\begin{theorem}
\label{thm:cs_generalization}
Let $\alpha\in[0,1/2)$ be arbitrary and let $K>0$. Fix $x\in\Sigma^N_k$ be such that $\min_{j\in\rm{supp}(x)}|x_j|\geq C\delta$, where $C$ depends only on $k$ and $K$. Suppose that $\Phi\in \R^{m\times N}$ whose entries are i.i.d. sub-Gaussian with sub-Gaussian norm not exceeding $K$, $k\geq 3$, and ${\lambda=m/k\geq c_4\log N}$ where $c_4=c_4(k,\alpha,K)$ is a constant. Then, with $x_{\rm MSQ}$ obtained using the two-stage method, there are constants $c_5,c_6,c_7>0$ that depend only on $k$, $K$, and $x$ such that with probability at least $1-c_5\exp(-c_6\lambda^{\alpha})$ on the draw of $\Phi$, the reconstruction error satisfies
$$\|x_{\rm MSQ}-x\|\leq A\left(\frac{1}{m}\|\mathbb{E}E^T(Ex-Q(Ex))\|+c_7\sqrt{\log{k}}\lambda^{-(1-\alpha)/2}\delta\right),$$
where $E=\Phi_T$, $T=\text{supp}(x)$, has i.i.d. sub-Gaussian entries, which are copies of the entries of $\Phi$.
\end{theorem}
Like before, we can calculate $\|\mathbb{E}E^T(Ex-Q(Ex))\|$ if $\Phi$, thus $E$, is a Gaussian random matrix. 
\begin{corollary}
	\label{thm:cs_generalization-gaussian}
	In the setting of Theorem \ref{thm:cs_generalization}, suppose that $\Phi\in \R^{m\times N}$ is a Gaussian matrix. Then with probability at least $1-c_5\exp(-c_6\lambda^{\alpha})$ on the draw of $\Phi$, the reconstruction error satisfies
	$$\|\widehat{x}_{\rm MSQ}-x\|\leq A\left(2\|x\|\exp(-\frac{2\pi^2\|x\|^2}{\delta^2})+c_7\sqrt{\log{k}}\lambda^{-(1-\alpha)/2}\delta\right).$$
Above the constants are as in Theorem \ref{thm:cs_generalization}.
\end{corollary}

\subsection{Projected back projection}
We carry the same notations as in the previous sections. The measurement matrix is denoted by $\Phi\in\mathbb{R}^{m\times N}$, and we assume that $\Phi$ is a sub-Gaussian random matrix with independent isotropic random rows. Consider the following reconstruction scheme for distorted measurements $q$:
$$x^{\text{PBP}} = H_k(\frac{1}{m}\Phi^Tq),$$
where $H_k$ is the projection onto the set of $k$-sparse signals. In other words, $H_k$ keeps the largest (in magnitude) $k$ entries of its argument and sets other entries to zero. We provide the bound for the reconstruction error
$$\|x^{\text{PBP}} - x\|=\|H_k(\frac{1}{m}\Phi^T Q(\Phi x)) - x\|.$$

\begin{theorem}
\label{thm:PBP_MSQ_bound:theorem}
Fix $x\in\Sigma^N_k$ with its support index set $T$. Suppose that $\Phi$ is a sub-Gaussian matrix with independent isotropic random rows whose sub-Gaussian norm does not exceed $K$. Assume that $\Phi$ satisfies the RIP of order $2k$ with constant $\delta_{2k}.$ Then, the reconstruction error satisfies
$$\|x^{\text{PBP}} - x\|\leq A\mu + C^{\text{PBP}}\sqrt{\log(k)}\lambda^{-1/2} + 4\delta_{2k}\|x\|$$
with probability $1 - c^{\text{PBP}} - e\exp(-c_3m).$ Here $\mu = \frac{1}{m}\|\mathbb{E}\Phi_T^T(\Phi_Tx - Q(\Phi_T x)\|,$ $A\sim 1+$ is the constant that depends on $\lambda = m/k$ only, $C^{\text{PBP}} = C(K\sqrt{ln(e^2(c^\text{PBP})^{-1})} + 1)$ where $C$ is an absolute constant.

If, in addition, $\Phi$ is a Gaussian matrix, the bound becomes
$$\|x^{\text{PBP}} - x\|\leq 2A\|x\|\exp(-\frac{2\pi^2\|x\|^2}{\delta^2}) + C^{\text{PBP}}\sqrt{\log(s)}\lambda^{-1/2} + 4\delta_{2k}\|x\|$$
with the same probability.
\end{theorem}

Let us compare this result with two related ones. First, note that if quantization is not applied, i.e., if $Q=I$, then $\|x^{\text{PBP}} - x\|\leq 2\delta_{3k} \|x\|$ (see, e.g., \cite{xu2019quantized}). If the quantization is dithered, then the result of Xu et al. \cite{xu2019quantized} implies that
$$\|x^{\text{dither PBP}} - x\|\leq 2\delta_{2k}(3+\delta)\|x\|.$$
Note that all three bounds contain the term of order $\delta_{2k} \|x\|$. Inverting constants in Theorem \ref{thm:subGaussian_RIP}, we get that for fixed $k$ and $N$, $\delta_{2k} = O(m^{-1/2})$, and therefore, in the case without quantization, and for the dithered quantization, the reconstruction error decays like $O(m^{-1/2}).$ Note that for the MSQ, the proven decay rate may be rewritten as
$$\|x^{\text{PBP}} - x\|\leq A\frac{1}{m}\|\mathbb{E}\Phi_T^T(\Phi_Tx - Q(\Phi_T x))\| + O(m^{-1/2}),$$
where the constant factor inside the big O depends on $s$, $N$, and $K$. Note that the first term is a constant if rows of $\Phi$ are i.i.d., but it is a tiny number, at least for Gaussian matrices.

\section{Preliminaries.}
\label{sec:preliminaries}
The following will be instrumental for the proofs, which we present in the next section. 
\begin{theorem}[Hoeffding inequality, see, e.g., { \cite[p. 220]{Vershbook}}]
\label{thm:Hoeffding}
Let $X_1$, $X_2$,..., $X_k$ be independent centered sub-Gaussian random variables, and let $K=\max_i\|X_i\|_{\psi_2}$. Then for every $a=(a_1,a_2,...,a_k)\in\mathbb{R}^N$ and every $t\geq 0$, we have
$$\mathbb{P}\left\{\left|\sum_{i=1}^k a_iX_i\right|\geq t\right\}\leq e\cdot\exp\left(-\frac{ct^2}{K^2\|a\|^2}\right),$$
where $c>0$ is an absolute constant.
\end{theorem}

Below $\sigma_{\max}(E)$ and $\sigma_{\min}(E)$ denote the largest and smallest singular values of a random matrix $E$. 

\begin{theorem}[{ \cite[p. 232]{Vershbook}}]
\label{thm:matrixnormestGeneral}
Suppose that $E$ is an $m\times k$ matrix with independent isotropic sub-Gaussian rows whose $\psi_2$-norm does not exceed $K$. Then, for every $t\geq 0$, with probability at least $1-2\exp(-c_3t^2)$,
$$\sqrt{m}-c_K\sqrt{k}-t\leq\sigma_{\min}(E)\leq\sigma_{\max}(E)\leq \sqrt{m}+c_K\sqrt{k}+t.$$
Here $c_K>0$ depends on $K$ only, and $c_3>0$ is an absolute constant. If $E$ is a Gaussian matrix, $c_K=1$. 
\end{theorem}

Plugging $t=\sqrt{m}/2$ in Theorem \ref{thm:matrixnormestGeneral} provides the following corollary.
\begin{corollary}
\label{cor:matrixnormest}
For a matrix $E$ that satisfies the conditions of Theorem \ref{thm:matrixnormestGeneral}, with probability at least $1-2\exp(-c_3m)$,
$$A'\frac{1}{m}\leq \sigma_{\min}(\left(E^TE\right)^{-1})\leq\sigma_{\max}(\left(E^TE\right)^{-1})\leq A\frac{1}{m},$$
where $A=(1/2-c_K\lambda^{-1/2})^{-2}>0$ and $A'=(3/2+c_K\lambda^{-1/2})^{-2}>0$. If $E$ is a Gaussian matrix, $c_K=1$. 

\end{corollary}

\begin{remark}
For the choice of $t=\sqrt{m}/2$ in Theorem \ref{thm:matrixnormestGeneral}, $\lambda$ should exceed $(2c_K)^{2}$, which is 4 for Gaussian case, to ensure that $1/2-c_K\lambda^{-1/2}>0$.
\end{remark}
The following theorem is a central tool for bounding the projected back projection (PBP) error.
\begin{theorem}[Lemma 2.1 in \cite{candes2008restricted}]
\label{thm:PBP_classical_result:theorem}
Suppose that $\frac{1}{\sqrt{m}}\Phi$ satisfies the RIP of order $\delta_{2k}$. Then, for every $x,\,\in\Sigma^N_k$, and $z\in\Sigma^N_{2k}$ whose support set includes the support set of $x$,
$$\left|\frac{1}{m}\langle \Phi x\,,\,\Phi z\rangle -\langle x\,,\, z\rangle\right|\leq 2\delta_{2k} \|x\|\|z\|.$$
Equivalently, if for any set $S$ such that $|S| = 2k$ and $\text{supp}(x)\subset S$, we restrict $\frac{1}{m}\Phi^T \Phi x$ to index set $S$, we get
$$\left|\left(\frac{1}{m}\Phi^T \Phi x - x\right)_S\right| \leq 2\delta_{2k}\|x\|.$$
\end{theorem}
We will use the following theorem to control the RIP constants of sub-Gaussian matrices.
\begin{theorem}[{ \cite[p. 254]{Vershbook}}]
\label{thm:subGaussian_RIP}
Let $\Psi$ be an $m\times N$ sub-Gaussian random matrix with independent isotropic rows whose $\psi_2$-norms do not exceed $K$ and let $\overline{\Psi}=\frac{1}{\sqrt{m}}\Psi$. For $\epsilon\in (0,1)$ such that $m\geq \widetilde{c}_4\epsilon^{-2}k\ln(eN/k)$, $\overline{\Psi}$ satisfies the RIP of order $k$ with constant $\delta_k\le \epsilon$, with probability at least $1-2\exp(-\widetilde{c}_K\epsilon^2m).$ Here $\widetilde{c}_4=\widetilde{c}_4(K)$ and $\widetilde{c}_K>0$ depend only on $K$.
\end{theorem}
Plugging $\epsilon=1/\sqrt{2}$ in the theorem above and using it for estimating the RIP constant $\delta_{2k}$ leads to the following corollary. Let $\bar{C}=4\widetilde{c}_4$ and $\bar{c}_K=\frac{1}{2}\widetilde{c}_K.$
\begin{corollary}
	\label{thm:cor:RIP_subGaussian}
Let $\Psi$ be an $m\times N$ sub-Gaussian random matrix with independent isotropic rows whose $\psi_2$-norms do not exceed $K$ and let $\overline{\Psi}=\frac{1}{\sqrt{m}}\Psi$. For $m\geq \bar{C} k\ln(0.5eN/k),$ $\overline{\Psi}$ satisfies the RIP of order $2k$ with constant $\delta_{2k}\le 1/\sqrt{2}$, with probability at least $1-2\exp(-\bar{c}_Km).$ Here $\bar{C}=\bar{C}(K)$ and $\bar{c}_K>0$ depend only on $K.$
\end{corollary}
We need to bound the mean of a function of sum of random variables. The Berry-Esseen theorem and the tail bound serves this purpose.
\begin{theorem}[Berry-Esseen theorem, see, e.g., \cite{shevtsova2010improvement}]
\label{thm:Berry-Esseen:theorem}
Let $X_i$, $i=1,2,...,k$ be independent random variables such that $\mathbb{E}X_i=0$, $\mathbb{E}X_i^2 = \sigma_i^2>0$ and $\mathbb{E}|X_i|^3 \leq \rho_i\leq \infty.$ Then, consider the distribution function of
$$\frac{X_1+X_2+...+X_k}{\sqrt{\sigma_1^2+\sigma_2^2+...+\sigma_k^2}}$$
denoted by $\Phi_k$ and the distribution function of the standard Gaussian random variable denoted by $\Psi$. Then, there exists constant $c_0$ such that
$$\sup_{t\in \mathbb{R}} |\Phi_k(t)-\Psi(t)| \leq c_0\frac{\sum_{i=1}^k \rho_i}{\left(\sum_{i=1}^k \sigma_k^2\right)^{3/2}}.$$
\end{theorem}

Finally, we also need the Poisson summation formula.
\begin{theorem}[{ \cite[p. 287]{epstein2007introduction}}]
Let $f:\mathbb{R}\mapsto\mathbb{R}$ be a Schwartz function. For any $a>0$ and $b\in\mathbb{R}$,
\begin{equation}
\label{eq:PoissonSumFormula}
\sum_{n\in\mathbb{Z}}f(an+b)=\sum_{p\in\mathbb{Z}}\frac{1}{a}\widehat{f}\left(\frac{p}{a}\right)\exp(2\pi i\frac{p}{a}b),
\end{equation}
where $\widehat{f}(\omega)=\int \exp(-2\pi i t\omega)f(t)dt$ is a Fourier transform of $f$.
\end{theorem}

\section{Proofs.}
\label{sec:proofs}
First, we provide a road map for obtaining the upper bound in Theorem \ref{thm:MSQforFrameGeneral}. Note that the reconstruction error satisfies
\begin{align}
\mathcal{E}(x)&=\|x-E^\dagger Q(Ex)\|=\|m\left(E^T E\right)^{-1}\cdot \frac{1}{m}E^T(Ex-Q(Ex))\| \notag \\
&\le \left(m\sigma_{\max}((E^TE)^{-1})\right)\left(\mu+\frac{1}{m}\|E^T(Ex-Q(Ex))-\mathbb{E}E^T(Ex-Q(Ex))\|\right), \label{eqproof1}
\end{align}
where $\mu =\frac{1}{m}\|\mathbb{E}E^T(Ex-Q(Ex))\|$. Here, the term $m\sigma_{\max}((E^TE)^{-1})$ is controlled using Corollary \ref{cor:matrixnormest}. In addition we need to estimate the deviation of $E^T(Ex-Q(Ex))$ from its mean $\mathbb{E}E^T(Ex-Q(Ex))$. This can be done entrywise using the following observations. Consider the $i$th entry of $\frac{1}{m}E^T(Ex-Q(Ex))$, say $\alpha_i$, given by
\begin{equation}
\label{exp:centered_entry_of_ETEx}
\alpha_i=\frac{1}{m}\sum_{j=1}^m\left( e_{ji}F(\sum_{s=1}^ke_{js}x_s)\right)
\end{equation}
where $F(z):=z-Q(z)$ for all $z\in\mathbb{R}$. Each summand in \eqref{exp:centered_entry_of_ETEx} depends only on the $j$th row of $E$. Thus, since $E$ has independent rows, the summands in \eqref{exp:centered_entry_of_ETEx} are independent. Furthermore, they are sub-Gaussian random variables whose sub-Gaussian norms do not exceed $K\delta/2$. So, we can  control the deviation of this sum from its mean as 
$$\Big |\frac{1}{m}\sum_{j=1}^m\left( e_{ji}F(\sum_{s=1}^ke_{js}x_s)-\mathbb{E}e_{ji}F(\sum_{s=1}^ke_{js}x_s)\right)\Big |\leq C_2\sqrt{\frac{1}{m}}$$
with probability at least $1-e\exp(-C_3C_2^2).$ Finally, $\frac{1}{m}\|E^T(Ex-Q(Ex))-\mathbb{E}E^T(Ex-Q(Ex))\|$ can be bounded using these entrywise estimates. Next we provide the full proof.

\subsection{Proof of Theorem \ref{thm:MSQforFrameGeneral}.}
Observe that $m>k$ implies that that $E$ is a full rank matrix with overwhelming probability, thus $E^\dagger$ is well-defined. Also by Corollary \ref{cor:matrixnormest}, 
$$A'\le \left(m\sigma_{\min}(E^TE)^{-1}\right)\le \left(m\sigma_{\max}(E^TE)^{-1}\right) \le A$$
with probability at least $1-2\exp(-c_3m)$, where $A$,$A'$, and $c_3$ are as in Corollary \ref{cor:matrixnormest}. This, together with \eqref{eqproof1}, gives 
$$
\mathcal{E}(x) \le A\left(\mu+\frac{1}{m}\|E^T(Ex-Q(Ex))-\mathbb{E}E^T(Ex-Q(Ex))\|\right)
$$
A similar calculation (using the reverse triangle inequality this time) yields the lower bound
$$
\mathcal{E}(x)
 \ge A'\left(\mu-\frac{1}{m}\|E^T(Ex-Q(Ex))-\mathbb{E}E^T(Ex-Q(Ex))\|\right).
$$
To finish the proof, we need to bound $\frac{1}{m}\|E^T(Ex-Q(Ex))-\mathbb{E}E^T(Ex-Q(Ex))\|=\big(\sum_{i=1}^k (\alpha_i-\mathbb{E}\alpha_i)^2\big)^{1/2}$, where $\alpha_i$ is as in \eqref{exp:centered_entry_of_ETEx}, from above and below. As outlined before, we will achieve this by controlling $\alpha_i-\mathbb{E}(\alpha_i)$ for each $i$. 

Observe that $\alpha_i$ is an average of $m$ terms ($e_{ji}F(\sum_{s=1}^k e_{js}x_s)$) that are
\emph{independent} because each term depends on $j$th row of matrix $E$ only, and according the conditions of the theorem, rows of $E$ are independent. Therefore, to bound $\alpha_i$ we use the Hoeffding inequality - see Lemma \ref{lem:CLT-weird} below. Combining the observations above and Lemma \ref{lem:CLT-weird} finishes the proof.

\begin{lemma}
\label{lem:CLT-weird}
Let $x\in\mathbb{R}^k$ and $E$ be an $m\times k$ matrix with independent rows such that its entries $e_{ij}$ are sub-Gaussian whose $\psi_2$-norms do not exceed $K$. Suppose that $\alpha_i$ is as in \eqref{exp:centered_entry_of_ETEx}. Then,
\begin{enumerate}[(i)]
\item Fix $1\leq i\leq k$. With any precision $c_2\in(0,1)$, with probability at least $1-\frac{1}{k}c_2$,
$$|\alpha_i-\mathbb{E}\alpha_i|\leq c_1\sqrt{\log{k}}\sqrt{\frac{1}{m}}K\delta.$$
\item With any precision $c_2\in(0,1)$, with probability at least $1-c_2$, we have
$$\frac{1}{m}\|E^T(Ex-Q(Ex))-\mathbb{E}E^T(Ex-Q(Ex))\|\leq c_1\sqrt{\log{k}}\sqrt{\frac{k}{m}}K\delta.$$
\end{enumerate}
Here $c_1=C\sqrt{\log(e^2c_2^{-1})}$ and $C$ is an absolute constant.
\end{lemma}

\begin{proof}[Proof of Lemma \ref{lem:CLT-weird}]
Note that 
$$\frac{1}{m}\|E^T(Ex-Q(Ex))-\mathbb{E}E^T(Ex-Q(Ex))\|=\Big(\sum_{i=1}^k(\alpha_i-\mathbb{E}\alpha_i)^2\Big)^{1/2}$$
and, therefore, the second statement of Lemma \ref{lem:CLT-weird} follows from the first one using straight-forward union bound on the first inequality. Therefore, we focus on the proof of the first statement of the lemma.

Fix $1\leq i\leq k$, and consider
$$
\alpha_i-\mathbb{E}\alpha_i=\frac{1}{m}\sum_{j=1}^m\left(e_{ji}F(\sum_{s=1}^ke_{js}x_s)-\mathbb{E}e_{ji}F(\sum_{s=1}^{k}e_{js}x_s)\right).
$$
Note that we aim to bound the sum of $m$ random variables $Z_j:=\displaystyle e_{ji}F(\sum_{s=1}^ke_{js}x_s)-\mathbb{E}e_{ji}F(\sum_{s=1}^{k}e_{js}x_s)$. 
\noindent{\bf Claim:} Given that $E$ has independent, isotropic, sub-Gaussian random rows, 
\begin{enumerate}[(a)]
\item $Z_j$ are independent. 
\item $Z_j$ are centered, sub-Gaussian random variables whose $\psi_2$-norm does not exceed $K\delta$. 
\end{enumerate}
\smallskip

Suppose this claim holds (it is proved below). Then, the Hoeffding inequality, as stated in Theorem \ref{thm:Hoeffding}, with ${t:=c_1\sqrt{\log k}\sqrt{m}K\delta}$, where $c_1>0$ is a constant to be determined later, implies that
\begin{equation}
\label{eq:HoeffdingInMyThm}
\mathbb{P}\left(\left|\alpha_i-\mathbb{E}\alpha_i\right|\leq c_1\sqrt{\log k}\sqrt{\frac{1}{m}}K\delta\right)\geq 1-e\exp\left(-cc_1^2\cdot\log(k)m\right),
\end{equation}
where $c>0$ is an absolute constant. Simplifying the probability of failure, we have
$$\mathbb{P}\left(\left|\alpha_i-\mathbb{E}\alpha_i \right|\leq c_1\sqrt{\log k}\sqrt{\frac{1}{m}}K\delta\right)\geq 1-ek^{-cc_1^2}.$$

One can simplify the expression for probability of failure further to obtain the statement in the lemma. Specifically,  if $cc_1^2>1$ (which can be guaranteed by choosing $c_2\in (0,1)$ and then picking $c_1$ such that $c_2=e^{2-cc^2_1}$),  then, for $k\geq 3$, the probability of failure may be bounded from above by $ek^{-cc_1^2} \le \frac{1}{k}c_2$.

The proof of the lemma will be complete once we prove the claim above. 

\noindent{\bf Proof of the claim.} Since the rows of $E$ are independent, so are $\{e_{ji}F(\sum_{s=1}^ke_{js}x_s)\}$, $j=1,2,..,m$, which, in turn, implies (a). Next note that $Z_j$ are centered by definition, so to prove (b), it suffices to show that ${e_{ji}F(\sum_{s=1}^ke_{js}x_s)}$ and ${\mathbb{E}e_{ji}F(\sum_{s=1}^ke_{js}x_s)}$ are sub-Gaussian. The latter term is a constant, and
\begin{align*}
\|Z_j\|_{\psi_2}&\leq\|e_{ji}F(\sum_{s=1}^ke_{js}x_s)\|_{\psi_2}+\big |\mathbb{E}\big[e_{ji}F(\sum_{s=1}^ke_{js}x_s)\big]\big |\\
&\leq\|e_{ji}F(\sum_{s=1}^ke_{js}x_s)\|_{\psi_2}+\mathbb{E}\big |e_{ji}F(\sum_{s=1}^ke_{js}x_s)\big |\leq 2\|e_{ji}F(\sum_{s=1}^ke_{js}x_s)\|_{\psi_2}\\
&\leq 2\|F(\sum_{s=1}^ke_{js}x_s)\|_\infty\|e_{ji}\|_{\psi_2}\leq2\cdot0.5\delta\|e_{ji}\|_{\psi_2}\leq\delta K\\
\end{align*}
Here we used basic properties of $\psi_2$-norm together with the fact that the range of $F$ is $(-\delta/2,\delta/2]$ and the sub-Gaussian norm of $e_{ji}$ does not exceed the sub-Gaussian norm of the $j$th row of $E$, which is at most $K$. 

This finishes the proof of the claim, consequently the proof of Lemma \ref{lem:CLT-weird}, and hence the proof of Theorem \ref{thm:MSQforFrameGeneral}.
\end{proof}

\subsection{Proof of Theorem \ref{thm:estmean}}
\begin{align*}
\frac{1}{m}\mathbb{E}\left(\sum_{j=1}^m e_{ji}F(\sum_{s=1}^ke_{js}x_s)\right)&= \mathbb{E}\left(e_{1i}F(\sum_{s=1}^k e_{1s}x_s)\right) \\
&= \mathbb{E}\left(e_{1i}(\sum_{s=1}^k e_{1s}x_s)\right)- \mathbb{E}\left(e_{1i}Q(\sum_{s=1}^k e_{1s}x_s)\right)\\
 &=\mathbb{E}x_ie_{1i}^2+\mathbb{E}\sum_{s\neq i}x_se_{1i}e_{1s}-\mathbb{E}\left(e_{1i}Q(\sum_{s=1}^k e_{1s}x_s)\right)\\
 &=x_i-\mathbb{E}\left(e_{1i}Q(\sum_{s=1}^k e_{1s}x_s)\right)\\
 \end{align*}
In the last equality we used the facts that $E$ has isotropic rows, so each entry has mean 0 and variance 1, and its entries 
are independent.

If $x_i=0$, then, using the fact that random variables $e_{1i}$ and $\sum_{s\neq i} e_{1s}x_s$ are independent,
$$\mathbb{E}\left(e_{1i}Q(\sum_{s=1}^k e_{1s}x_s)\right)=\mathbb{E}e_{1i}\mathbb{E}Q(\sum_{s=1}^k e_{1s}x_s)=0.$$
Then, $\mathbb{E}e_{1i}F(\sum_{s=1}^k e_{1s}x_s)=x_i=0$, and the statement of the theorem holds.

Now assume that $x_i\neq 0$. There are two possibilities, namely, $x_i>0$ and $x_i<0$. We join them into the one case for a potentially different set of $x_i$'s using the oddity of $Q$. Let $\overline{x_s}:=x_s\text{sign}(x_i),$ $s=1,2,...,k$. Then,
\begin{equation}
\label{eq:ComputeMean_F_to_Q}
\mathbb{E}\left(e_{1i}F(\sum_{s=1}^k e_{1s}x_s)\right)=x_i-\text{sign}(x_i)\mathbb{E}\left(e_{1i}Q(\sum_{s=1}^k e_{1s}\overline{x_s})\right).
\end{equation}
Note that $\overline{x_i}=x_i\text{sign}(x_i)=|x_i|$ which is positive because $x_i\neq 0$. Therefore, it suffices to estimate $\mathbb{E}\big(e_{1i}Q(\sum_{s=1}^k e_{1s}\overline{x_s})\big)$.
 
Let $\one$ be an indicator function of a Boolean variable, taking the value 1 if its argument is true and 0 otherwise. Using this notation and the definition of $Q$, we have
\begin{equation}\label{eq:pf:111}
\mathbb{E}\left(e_{1i}Q(\sum_{s=1}^k e_{1s}\overline{x_s})\right)=\sum_{n\in\mathbb{Z}}\mathbb{E}e_{1i}\cdot n\delta\cdot\one\{\sum_{s=1}^ke_{1s}\overline{x_s}\in (n\delta-\frac{\delta}{2},n\delta+\frac{\delta}{2}]\}.
\end{equation}
Let $\eta_i:=\sum_{s\neq i}e_{1s}\overline{x_s}.$ Note that 
$$\one\{\sum_{s=1}^ke_{1s}\overline{x_s}\in (n\delta-\frac{\delta}{2},n\delta+\frac{\delta}{2}]\}=\begin{cases}
1,& \text{if } e_{1i}\in(\frac{n\delta-\delta/2-\eta_i}{\overline{x_i}},\frac{n\delta+\delta/2-\eta_i}{\overline{x_i}}]\\
0,& \text{otherwise}
\end{cases}.$$
Therefore, one may rewrite the summands on the right hand side of \eqref{eq:pf:111} as 
\begin{align}
\label{eq:split_chi}
\mathbb{E}e_{1i}n\delta\one\{\sum_{s=1}^ke_{1s}x_s&\in (n\delta-\frac{\delta}{2},n\delta+\frac{\delta}{2}]\} \\
&=\mathbb{E}e_{1i}n\delta\one\{e_{1i}\leq\frac{n\delta+\delta/2-\eta_i}{\overline{x_i}}\}-\mathbb{E}e_{1i}n\delta\one\{e_{1i}\leq\frac{n\delta-\delta/2-\eta_i}{\overline{x_i}}\}. \notag
\end{align}		

Next, let $g(z):=\mathbb{E}e_{1i}\one\{e_{1i}\leq z\}$ for $z\in\mathbb{R}$. Since all entries of $E$ are distributed identically, $g$ does not depend on $i$. We claim that $g$ is a Schwartz function. Indeed,
$$g(z)=\int_{-\infty}^z t\phi(t)dt,$$
where $\phi$ is the density function of $e_{11}$. Note that by our hypothesis $\phi$ is a Schwartz function. So, clearly, $g\in\mathcal{C}^\infty$. In addition, since $\phi(t)=O(|t|^{-n-2})$ for all $n$ as $t\rightarrow -\infty$, $g(z)=O(|t|^{-n})$ for all $n$ as $t\rightarrow -\infty$. Now recall that the rows of $E$ are isotropic, so the mean value of each entry is 0. Then,
$$g(z)=\mathbb{E}e_{11}-\int_z^\infty t\phi(t)dt=-\int_z^\infty t\phi(t)dt.$$
By similar arguments, $g(z)=O(|t|^{-n})$ as $z\rightarrow\infty.$ Therefore, $g$ is a Schwartz function.

In what follows, we will have random arguments for $g$. Specifically, for a random variable $\eta$ that is independent of $e_{1i}$, let
$$g(\eta):=\mathbb{E}\left[e_{1i}\one\{e_{1i}\leq\eta\}\,|\, \eta\right].$$
Note that this definition of $g$ agrees with the previous one when one takes $\eta=z$ a.e. (which is independent of $e_{1i}$). Using the law of total expectation for each term in (\ref{eq:split_chi}) and the definition of $g$, we get
\begin{align*}\mathbb{E}e_{1i}n\delta\one\{e_{1i}\leq\frac{n\delta+\delta/2-\eta_i}{\overline{x_i}}\}&=\mathbb{E}\left[\mathbb{E}\left[e_{1i}n\delta\one\{e_{1i}\leq\frac{n\delta+\delta/2-\eta_i}{\overline{x_i}}\}\,|\, \eta_i\right]\right]\\
&=n\delta\mathbb{E}g(\frac{n\delta+\delta/2-\eta_i}{\overline{x_i}})\\
\end{align*}
and, similarly,
$$\mathbb{E}e_{1i}n\delta\one\{e_{1i}\leq\frac{n\delta-\delta/2-\eta_i}{\overline{x_i}}\}=n\delta\mathbb{E}g(\frac{n\delta-\delta/2-\eta_i}{\overline{x_i}}).$$
Therefore, (\ref{eq:split_chi}) may be rewritten as follows,
$$\mathbb{E}e_{1i}n\delta\one\{\sum_{s=1}^ke_{1s}\overline{x_s}\in (n\delta-\delta/2,n\delta+\delta/2]\}=n\delta\mathbb{E}g(\frac{n\delta+\delta/2-\eta_i}{\overline{x_i}})-n\delta\mathbb{E}g(\frac{n\delta-\delta/2-\eta_i}{\overline{x_i}}),$$
which, in turn, allows us to rewrite the original expression $\mathbb{E}\left(e_{1i}Q(\sum_{s=1}^k e_{1s}\overline{x_s})\right)$ as
\begin{align*}\mathbb{E}\left(e_{1i}Q(\sum_{s=1}^k e_{1s}\overline{x_s})\right)&=\mathbb{E}\sum_{n\in\mathbb{Z}}\left(n\delta(g(\frac{n\delta+\delta/2-\eta_i}{\overline{x_i}})-g(\frac{n\delta-\delta/2-\eta_i}{\overline{x_i}}))\right)\\
&=\mathbb{E}\sum_{n\in\mathbb{Z}}\left(n\delta g(\frac{n\delta+\delta/2-\eta_i}{\overline{x_i}})-((n-1)+1)\delta g(\frac{(n-1)\delta+\delta/2-\eta_i}{\overline{x_i}}))\right)\\
\end{align*}
Splitting the series and shifting the index of summation by 1 leads to
\begin{equation}
\label{eq:before_poisson}
\mathbb{E}\left(e_{1i}Q(\sum_{s=1}^k e_{1s}\overline{x_s})\right)=\mathbb{E}\left(-\delta\sum_{n\in\mathbb{Z}}g(\frac{n\delta+\delta/2-\eta_i}{\overline{x_i}})\right).
\end{equation}
Here splitting the series can be justified by using Fubini-Tonelli theorem twice together with the fact that $g$ is a Schwartz function, which implies $g(z)=O(|z|^{-3})$ as $|z|\rightarrow\infty$. Therefore, $\sum_{n\in\mathbb{Z}}n\delta g(n\delta\pm\delta/2-\eta_i)$ is convergent for any value of $\eta_i$ in $\mathbb{R}$.

Next, we will estimate the mean value of the series, where we use the Poisson summation formula (\ref{eq:PoissonSumFormula}), which may be applied as $g$ is a Schwartz function. Denote the Fourier transform of $g$ by $\widehat{g}(\omega):=\int \exp(-i2\pi t\omega)g(t)dt$. Then,
\begin{align*}
\sum_{n\in\mathbb{Z}}g(n\frac{\delta}{\overline{x_i}}+\frac{\delta/2-\eta_i}{\overline{x_i}})&=\sum_{p\in\mathbb{Z}}\frac{\overline{x_i}}{\delta}\widehat{g}(\frac{\overline{x_i}}{\delta}p)\exp(i2\pi\frac{\overline{x_i}p}{\delta}\frac{\delta/2-\eta_i}{\overline{x_i}})\\
&=\frac{\overline{x_i}}{\delta}\sum_{p\in\mathbb{Z}}(-1)^p\widehat{g}(\frac{\overline{x_i}}{\delta}p)\exp(-i2\pi\frac{p}{\delta}\eta_i)\\
\end{align*}
Plugging the last formula into (\ref{eq:before_poisson}) gives us the following equation:
$$\mathbb{E}\left(e_{1i}Q(\sum_{s=1}^k e_{1s}\overline{x_s})\right)=\mathbb{E}\left(-\delta\frac{\overline{x_i}}{\delta}\sum_{p\in\mathbb{Z}}(-1)^p\widehat{g}(\frac{\overline{x_i}}{\delta}p)\exp(-i2\pi\frac{p}{\delta}\eta_i)\right).$$
We want to interchange the expectation and infinite sum. To do this, we need to verify the conditions of dominated convergence theorem. Note that
$$|(-1)^p\widehat{g}(\frac{\overline{x_i}}{\delta}p)\exp(-i2\pi\frac{p}{\delta}\eta_i)|=|\widehat{g}(\frac{\overline{x_i}}{\delta}p)|.$$Since $g$ is a Swartz function, so is $\widehat{g}$, and therefore,
$$\mathbb{E}\sum_{p\in\mathbb{Z}}|(-1)^p\widehat{g}(\frac{\overline{x_i}}{\delta}p)\exp(-i2\pi\frac{p}{\delta}\eta_i)|=\mathbb{E}\sum_{p\in\mathbb{Z}}|\widehat{g}(\frac{\overline{x_i}}{\delta}p)|=\sum_{p\in\mathbb{Z}}|\widehat{g}(\frac{\overline{x_i}}{\delta}p)|$$
is convergent. Note that the last equality holds because we consider the expected value of a deterministic expression, so the expected value is redundant. Using Fubini-Tonelli theorem, we interchange the expected value and the infinite sum, and obtain
\begin{equation}
\label{eq:after_poisson}
\mathbb{E}\left(e_{1i}Q(\sum_{s=1}^k e_{1s}\overline{x_s})\right)=\left(-\delta\frac{\overline{x_i}}{\delta}\sum_{p\in\mathbb{Z}}\left((-1)^p\widehat{g}(\frac{\overline{x_i}}{\delta}p)\mathbb{E}\exp(-i2\pi\frac{p}{\delta}\eta_i)\right)\right).
\end{equation}
Moreover, since $\eta_i=\sum_{s\neq i}e_{1s}\overline{x_s}$ and $\{e_{1s}\}$ is a collection of independent random variables,
\begin{align*}
\mathbb{E}\exp(-i2\pi\frac{p}{\delta}\eta_i)&=\mathbb{E}\prod_{s\neq i}\exp(-i2\pi\frac{p}{\delta}\overline{x_s}e_{1s})\\
&=\prod_{s\neq i}\mathbb{E}\exp(-i2\pi\frac{p}{\delta}\overline{x_s}e_{1s})\\
&=\prod_{s\neq i}\widehat{\phi}(\frac{p}{\delta}\overline{x_s})\\
\end{align*}
where $\phi$ is the density function of $e_{1s}$ (recall that all entries of matrix $E$ are identically distributed).
Plugging the result above into (\ref{eq:after_poisson}) and then plugging into (\ref{eq:ComputeMean_F_to_Q}), we get
\begin{align*}
\mathbb{E}\left(e_{1i}F(\sum_{s=1}^k e_{1s}x_s)\right)&=x_i-\text{sign}(x_i)\left(-\overline{x_i}\sum_{p\in\mathbb{Z}}\left((-1)^p\widehat{g}(\frac{\overline{x_i}}{\delta}p)\prod_{s\neq i}\widehat{\phi}(\frac{p}{\delta}\overline{x_s})\right)\right),\\
&=x_i+x_i\sum_{p\in\mathbb{Z}}\left((-1)^p\widehat{g}(\frac{|x_i|}{\delta}p)\prod_{s\neq i}\widehat{\phi}(\frac{x_s\text{sign}(x_i)}{\delta}p)\right)
\end{align*}
where $\phi$ is a density function of $e_{11}$ and $g(z)=\mathbb{E}e_{11}\one\{e_{11}<z\}=\int_{-\infty}^z t\phi(t)dt.$ To finish the proof of the first statement of the Theorem,
\begin{align*}
\mathbb{E}\left(\sum_{j=1}^m e_{ji}F(\sum_{s=1}^ke_{js}x_s)\right)&=m\mathbb{E}\left(e_{1i}F(\sum_{s=1}^k e_{1s}x_s)\right)\\
&=m\left(x_i+x_i\sum_{p\in\mathbb{Z}}\left((-1)^p\widehat{g}(\frac{|x_i|}{\delta}p)\prod_{s\neq i}\widehat{\phi}(\frac{x_s\text{sign}(x_i)}{\delta}p)\right)\right)\\
\end{align*}
which coincides with (\ref{eq:OneEntry_Schwartz}).

Next, assume that entries of $E$ are i.i.d. standard Gaussian random variables. Then, the density function $\phi$ and its Fourier transform  $\widehat{\phi}$ are given by 
$$\phi(z)=\sqrt{\frac{1}{2\pi}}\exp\left(-\frac{z^2}{2}\right), \quad \widehat{\phi}(\omega)=\exp(-2\pi^2\omega^2).$$
Consequently,
$$g(z):=\int_{-\infty}^z t\phi(t)dt=\int_{-\infty}^z t\cdot\sqrt{\frac{1}{2\pi}}\exp(-t^2/2)dt=-\sqrt{\frac{1}{2\pi}}e^{-\frac{z^2}{2}}.$$
and its Fourier transform is $\widehat{g}(\omega):=-\exp(-2\pi^2\omega^2)$.
Plugging the exact values of $\widehat{\phi}$ and $\widehat{g}$ into (\ref{eq:OneEntry_Schwartz}),
\begin{align*}
\mathbb{E}\left(\sum_{j=1}^m e_{ji}F(\sum_{s=1}^ke_{js}x_s)\right)&=m\left(x_i+x_i\sum_{p\in\mathbb{Z}}\left((-1)^p\left(-e^{-2\pi^2\frac{x_i^2p^2}{\delta^2}}\right)\prod_{s\neq i}e^{-2\pi^2\frac{x_s^2p^2}{\delta^2}}\right)\right)\\
&=m\left(x_i+x_i\sum_{p\in\mathbb{Z}}\left(-(-1)^pe^{-2\pi^2\frac{\sum_{s=1}^k x_s^2}{\delta^2}p^2}\right)\right)\\
&=m\left(x_i-x_i\sum_{p\in\mathbb{Z}}(-1)^pe^{-2\pi^2\frac{\|x\|^2}{\delta^2}p^2}\right).\\
\end{align*}
Note that terms of the series are even with respect to $p$, i.e., the value of the terms for $p$ and $-p$ are identical. When $p=0$, the term is $1$. Therefore, one may simplify the series as follows.
\begin{align*}
\mathbb{E}\left(\sum_{j=1}^m e_{ji}F(\sum_{s=1}^ke_{js}x_s)\right)&=m(x_i-x_i-2x_i\sum_{p=1}^\infty(-1)^pe^{-2\pi^2\frac{\|x\|^2}{\delta^2}p^2})\\
&=-2mx_i\sum_{p=1}^\infty (-1)^pe^{-2\pi^2\frac{\|x\|^2}{\delta^2}p^2}.\\
\end{align*}
This is an alternating series, so one may bound the sum using the first two terms of the series. To be more precise, the exact value of the series $\mathbb{E}\left(\sum_{j=1}^me_{ji}F(\sum_{s=1}^ke_{js}x_s)\right)$ lies between
$2mx_i\exp(-2\pi^2\frac{\|x\|^2}{\delta^2})$ and $2mx_i\exp(-2\pi^2\frac{\|x\|^2}{\delta^2})-2mx_i\exp(-8\pi^2\frac{\|x\|^2}{\delta^2})$, which implies \eqref{eq:single_entry_Gaussian_est} and finishes the proof.
\qed

\subsection{Proof of Theorem \ref{thm:value_of_mu_bound:theorem}}
The rigorous proof is based on the heuristics outlined in Section \ref{sec:on_the_value_of_mu:subsection}. We introduce the following notations for simplicity: for unit-norm $x\in\mathbb{R}^k$, let $z=e_1^T x$ and $\xi\sim\mathcal{N}(0,1)$. Note that $\mathbb{E}z^2 = 1$. Then, using the Cauchy-Schwarz inequality and the inverse triangle inequality,
\begin{align*}
\mu &\geq \frac{\left|\mathbb{E}e_1^TxF(e_1^Tx)\right|}{\|x\|}=|\mathbb{E}zF(z)|\geq |\mathbb{E}\xi F(\xi)| - |\mathbb{E}\xi F(\xi) - \mathbb{E} zF(z)|\\
& = |\mathbb{E}\xi F(\xi)| - |\mathbb{E}\xi^2 - \mathbb{E}\xi Q(\xi) - \mathbb{E}z^2 + \mathbb{E} zQ(z)| = |\mathbb{E}\xi F(\xi)| - |\mathbb{E} zQ(z) - \mathbb{E}\xi Q(\xi)|.
\end{align*}
The first term is computed in Theorem \ref{thm:estmean} as follows.
$$2\left(e^{-\frac{2\pi^2}{\delta^2}}-e^{-\frac{8\pi^2}{\delta^2}}\right)\leq|\mathbb{E}\xi F(\xi)|=2\sum_{p=1}^\infty (-1)^{p} e^{-\frac{2\pi^2 p^2}{\delta^2}}\leq 2e^{-\frac{2\pi^2}{\delta^2}}.$$
We conclude that
\begin{equation}
\label{eq:reconst_error_as_k_grows_lower_bound:inequality}
\mu \geq 2\left(e^{-\frac{2\pi^2}{\delta^2}}-e^{-\frac{8\pi^2}{\delta^2}}\right)- |\mathbb{E} zQ(z) - \mathbb{E}\xi Q(\xi)|. 
\end{equation}

Next, we show that for the broad class of signals $x\in\R^k$, $|\mathbb{E} zQ(z) - \mathbb{E}\xi Q(\xi)|$ is arbitrarily small for large enough $k$. Note that for i.i.d. sub-Gaussian $e_{ij}$ and broad class of signals $x^k$, $z:=\sum_{s=1}^k e_{1s}x_s$ converges to a standard Gaussian random variable in distribution, and below, we establish the non-asymptotic version of it.

Fix $C\in\mathbb{N}$. Then, using that $zQ(z)\one_{z\in(-\delta/2,\delta/2)}= \mathbb{E}\xi Q(\xi)\one_{\xi\in(-\delta/2,\delta/2]}=0$ and the triangle inequality, we conclude that
\begin{align*}
\left|\mathbb{E} zQ(z) - \mathbb{E}\xi Q(\xi)\right| &\leq \left|\mathbb{E}zQ(z)\one_{z\in(\delta/2,C\delta +\delta/2]} -\mathbb{E}\xi Q(\xi)\one_{\xi\in(\delta/2, C\delta/2 +\delta/2]}\right| \\
& + \left|\mathbb{E}zQ(z)\one_{z\in(-C\delta/2 -\delta/2,-\delta/2]} -\mathbb{E}\xi Q(\xi)\one_{\xi\in(-C\delta/2-\delta/2,-\delta/2]}\right|\\
& +\left|\mathbb{E}zQ(z)\one_{z \leq -C\delta/2 -\delta/2} -\mathbb{E}\xi Q(\xi)\one_{\xi \leq -C\delta/2 -\delta/2} +\mathbb{E}zQ(z)\one_{z > C\delta/2 + \delta/2} - \mathbb{E}\xi Q(\xi)\one_{\xi > C\delta/2 +\delta/2}\right|.
\end{align*}
Recall that for any non-negative random variable, $\mathbb{E}\gamma = \int_0^\infty \mathbb{P}(\gamma > t) dt.$ Using this formula and that $\gamma Q(\gamma) \geq 0$ for all $\gamma$, we get
\begin{align*}
\left|\mathbb{E} zQ(z) - \mathbb{E}\xi Q(\xi)\right| &\leq \left|\int_0^\infty \left(\mathbb{P}\left(zQ(z)\one_{z\in(\delta/2,C\delta +\delta/2]} > t\right) - \mathbb{P}\left(\xi Q(\xi)\one_{\xi\in(\delta/2,C\delta + \delta/2]}>t\right)\right)dt\right|\\
& + \left|\int_0^\infty \left(\mathbb{P}\left(zQ(z)\one_{z\in(-C\delta - \delta/2,-\delta/2]} > t\right) - \mathbb{P}\left(\xi Q(\xi)\one_{\xi\in(-C\delta/2 -\delta/2,-\delta/2]}>t\right)\right)dt\right|\\
& + \left|\int_0^\infty \left[\mathbb{P}\left(zQ(z)\one_{|z|\geq C\delta +\delta/2} > t\right) - \mathbb{P}\left(\xi Q(\xi)\one_{|\xi|\geq C\delta + \delta/2} > t\right)\right]dt\right|.
\end{align*}
Note that on each interval $(i\delta - \delta/2, i\delta + \delta/2]$, $Q$ is constant. Then, we can rewrite sums as follows.
\begin{align*}
\left|\mathbb{E} zQ(z) - \mathbb{E}\xi Q(\xi)\right| &\leq \left|\int_0^\infty \sum_{i=1}^C\left(\mathbb{P}\left(z i\delta\one_{z\in(i\delta-\delta/2,i\delta +\delta/2]} > t\right) - \mathbb{P}\left(\xi i\delta)\one_{\xi\in(i\delta -\delta/2,i\delta + \delta/2]}>t\right)\right)dt\right|\\
& + \left|\int_0^\infty \sum_{i=1}^C\left(\mathbb{P}\left(z(-i\delta)\one_{z\in(-i\delta - \delta/2,-i\delta +\delta/2]} > t\right) - \mathbb{P}\left(\xi (-i\delta)\one_{\xi\in(-i\delta -\delta/2,-i\delta + \delta/2]}>t\right)\right)dt\right|\\
& + \left|\int_{0}^\infty \left[\mathbb{P}\left(|z|\geq C\delta +\delta/2\text{ and }zQ(z) > t\right) - \mathbb{P}\left(|\xi|\geq C\delta + \delta/2\text{ and }\xi Q(\xi) > t\right)\right]dt\right|.
\end{align*}
Note that if $|z|>\delta/2$, then $zQ(z) > (|z| -\delta/2)^2$. The same observation applies to $\xi Q(\xi)$ for $|\xi|>\delta/2$. Then,
\begin{align*}
\left|\mathbb{E} zQ(z) - \mathbb{E}\xi Q(\xi)\right| &\leq \left|\int_0^\infty \sum_{i=1}^C\left(\mathbb{P}\left(\max\{\frac{t}{i\delta},\, i\delta -\frac{\delta}{2}\} < z \leq i\delta +\frac{\delta}{2}\right) - \mathbb{P}\left(\max\{\frac{t}{i\delta},\, i\delta -\frac{\delta}{2}\} < \xi \leq i\delta +\frac{\delta}{2}\right)\right)dt\right|\\
& + \left|\int_0^\infty \sum_{i=1}^C\left(\mathbb{P}\left(-i\delta - \frac{\delta}{2} < z \leq \min\{-\frac{t}{i\delta},\,-i\delta +\frac{\delta}{2}\}\right) \right.\right.\\
& \left.\left.- \mathbb{P}\left(-i\delta -\frac{\delta}{2} < \xi\leq \min\{-\frac{t}{i\delta},\,-i\delta + \delta/2\}\right)\right)dt\right|\\
& + \left|\int_{0}^\infty \left[\mathbb{P}\left(|z|\geq C\delta +\delta/2\text{ and }(|z| -\frac{\delta}{2})^2 > t\right) - \mathbb{P}\left(|\xi|\geq C\delta + \delta/2\text{ and }(|\xi|-\frac{\delta}{2})^2  > t\right)\right]dt\right|.
\end{align*}
Note that for large values of $t$, $\frac{t}{i\delta} \geq i\delta +\frac{\delta}{2}$ and $-i\delta -\delta/2 \geq -\frac{t}{i\delta}$. In this case, the expression under the corresponding integral becomes zero. Then, using the triangle inequality, we conclude that
\begin{align*}
\left|\mathbb{E} zQ(z) - \mathbb{E}\xi Q(\xi)\right| &\leq \sum_{i=1}^C\int_0^{i(i+\frac{1}{2})\delta^2}\left|\mathbb{P}\left(\max\{\frac{t}{i\delta},\, i\delta -\frac{\delta}{2}\} < z \leq i\delta +\frac{\delta}{2}\right) - \mathbb{P}\left(\max\{\frac{t}{i\delta},\, i\delta -\frac{\delta}{2}\} < \xi \leq i\delta +\frac{\delta}{2}\right)\right|dt\\
& + \sum_{i=1}^C\int_0^{i(i+\frac{1}{2})\delta^2}\left|\mathbb{P}\left(-i\delta - \frac{\delta}{2} < z \leq \min\{-\frac{t}{i\delta},\,-i\delta +\frac{\delta}{2}\}\right) \right.\\
& \left.- \mathbb{P}\left(-i\delta -\frac{\delta}{2} < \xi\leq \min\{-\frac{t}{i\delta},\,-i\delta + \delta/2\}\right)\right|dt\\
& +\int_{0}^{C^2\delta^2} \left|\mathbb{P}\left(|z|\geq C\delta +\delta/2\right) - \mathbb{P}\left(|\xi|\geq C\delta + \delta/2\right)\right|dt\\
& +\int_{C^2\delta^2}^\infty \left|\mathbb{P}\left(|z|\geq \sqrt{t} +\delta/2\right) - \mathbb{P}\left(|\xi|\geq \sqrt{t} + \delta/2\right)\right|dt.
\end{align*}
The Berry-Esseen Theorem (Theorem \ref{thm:Berry-Esseen:theorem}) provides a quantitative bound on the difference of probabilities. For the setting above, the following corollary holds.
\begin{corollary}
\label{thm:Berry_Esseen:corollary}
Suppose that $\|x\|=1$, and let $X_i=e_{1i}x_i.$ Then, Theorem \ref{thm:Berry-Esseen:theorem} implies that for each $t\in\mathbb{R},$
$$\left|\mathbb{P}\left(\sum_{i=1}e_{1i}x_i  < t\right)-\mathbb{P}\left(\xi<t\right)\right|
\leq c_0 \sum_{i=1}^k |x_i|^3 = c_0 \|x\|_3^3.$$
\end{corollary}

Using Corollary \ref{thm:Berry_Esseen:corollary}, we get
\begin{align*}
\left|\mathbb{E} zQ(z) - \mathbb{E}\xi Q(\xi)\right| &\leq \sum_{i=1}^C\int_0^{i(i+\frac{1}{2})\delta^2}2c_0\|x\|_3^3dt + \sum_{i=1}^C\int_0^{i(i+\frac{1}{2})\delta^2}2c_0\|x\|_3^3dt\\
& + \int_{0}^{C^2\delta^2} c_0\|x\|^3_3dt\\
& +\int_{C^2\delta^2}^\infty \left|\mathbb{P}\left(|z|\geq \sqrt{t} +\delta/2\right) - \mathbb{P}\left(|\xi|\geq \sqrt{t} + \delta/2\right)\right|dt\\
&=4c_0\|x\|_3^3\left(\frac{C(C+1)(2C+1)}{6} + \frac{C(C+1)}{4}\right)\delta^2 + c_0\|x\|^3_3 C^2\delta^2\\
& +\int_{C^2\delta^2}^\infty \left|\mathbb{P}\left(|z|\geq \sqrt{t} +\delta/2\right) - \mathbb{P}\left(|\xi|\geq \sqrt{t} + \delta/2\right)\right|dt.\\
&=\frac{1}{3}c_0\|x\|_3^3 C(2C+1)(2C+5)\delta^2 +\int_{C^2\delta^2}^\infty \left|\mathbb{P}\left(|z|\geq \sqrt{t} +\delta/2\right) - \mathbb{P}\left(|\xi|\geq \sqrt{t} + \delta/2\right)\right|dt.\\
\end{align*}
Finally, we bound the integral using the fact that both $z$ and $\xi$ are sub-Gaussian random variables. Using the triangle inequality and that $\sqrt{t} +\delta/2 \geq \sqrt{t}$, the Hoeffding inequality (Theorem \ref{thm:Hoeffding}) and the Chernoff inequality, we get the following bound.
\begin{align*}
\int_{C^2\delta^2}^\infty \left|\mathbb{P}\left(|z|\geq \sqrt{t} +\delta/2\right) - \mathbb{P}\left(|\xi|\geq \sqrt{t} + \delta/2\right)\right|dt
&\leq \int_{C^2\delta^2}^\infty \mathbb{P}\left(|z|\geq \sqrt{t}\right)dt + \int_{C^2\delta^2}^\infty\mathbb{P}\left(|\xi|\geq \sqrt{t}\right)dt\\
&\leq \int_{C^2\delta^2}^\infty e\exp\left(-\frac{ct}{K^2}\right)dt + \int_{C^2\delta^2}^\infty 2\exp\left(-\frac{t}{2}\right)dt\\
& = eK^2\exp\left(-\frac{cC^2\delta^2}{K^2}\right) + 4\exp\left(-\frac{C^2\delta^2}{2}\right).\\
\end{align*}
Combining the bounds finishes the proof.
\qed
\subsection{Proof of Theorem \ref{thm:dithered_quantization:theorem}}
The proof closely resembles the proof of Theorem \ref{thm:MSQforFrameGeneral}. Expanding the reconstruction error, we get
\begin{align*}
\|x - x^{\text{dither}}\|&=\|x - E^\dagger Q(Ex+\tau)\|=\|(E^TE)^{-1}E^T\left(Ex - Q(Ex+\tau)\right)\\
& \leq \frac{1}{\sigma_{\min}^2(E)}\|E^T\left(F(Ex+\tau)-\tau\right)\|.
\end{align*}
Here $\sigma_{\min}$ denotes the smallest singular value of $E$, and $F:\mathbb{R}\mapsto(-\delta/2,\delta/2]$ is defined by $F(z) := z - Q(z)$. Recall that $e_1^T, e_2^T, ..., e_m^T$ denote rows of $E$, and assume that $\|e_{1}\|_{\psi_2}\leq K$. Then,
$$E^T\left(F(Ex+\tau)-\tau\right)=\sum_{i=1}^m e_i \left(F(e_1^Tx+\tau_i)-\tau_i\right).$$
Note that for $i\neq j$, $(e_i, \tau_i)$ is independent of $(e_j, \tau_j),$ which, in turn, implies that all terms in the sum are i.i.d. random vectors. We claim that $e_1(F(e_1^Tx+\tau_1)-\tau_1)$ is a sub-Gaussian random vector whose sub-Gaussian norm does not exceed $\delta K$. Indeed, note that the range of $F$ is bounded by $\delta/2$ in magnitude, and $\tau$ is bounded by $\delta/2$. Therefore, $|F(e_1^Tx+\tau_1)-\tau_1|\le \delta$, and 
$$\|e_1\left(F(e_1^Tx+\tau_1)-\tau_1\right)\|_{\psi_2}\leq\|e_1\|_{\psi_2}\delta\leq K\delta.$$

The sum of i.i.d. sub-Gaussian random variables can be bounded using the Hoeffding inequality (Theorem \ref{thm:Hoeffding}), and this result can be extended to random vectors by considering entry-wise bounds. Let us compute the mean of random variables.
\begin{align*}
\mathbb{E}E^T\left(F(Ex+\tau_1)-\tau_1\right)&= m\mathbb{E}e_1(F(e_1^T x+\tau_1)-\tau_1)\\
&= m\mathbb{E}e_1F(e_1^T x + \tau_1) - m\mathbb{E}e_1\tau_1\\
&= m\mathbb{E}_{e_1} e_1\mathbb{E}_{\tau_1} F(e_1^T x+\tau_1) - m \mathbb{E}_{e_1}e_1 \mathbb{E}_{\tau_1}\tau_1 = 0 - 0 = 0.\\
\end{align*}
Here, $\mathbb{E}_{e_1}$ and $\mathbb{E}_{\tau}$ stand for the expected values with respect to $e_1$ and $\tau$, respectively. In the last line, we used that $\mathbb{E}\tau_1 = 0$ and that for every $z$, $\mathbb{E}F(z+\tau_1)=0.$

The Hoeffding inequality (Theorem \ref{thm:Hoeffding}) implies that
$$\mathbb{P}\left(\left\|\frac{1}{m}E^T\left(F(Ex+\tau_1)-\tau_1\right)\right\|^2 > t^2\right)\leq ek\exp\left(-\frac{ct^2}{\delta^2}\right).$$
Combining this inequality with the bound on $\sigma_{\min}(A)$ provided by Theorem \ref{thm:matrixnormestGeneral}, we finish the proof.
\qed

\subsection{Proof of Theorem \ref{thm:cs_generalization}}
The proof is similar to the proof of Theorem B in  \cite{Sobolev_Duals_for_RF}. In this section, we only reproduce the main ideas of the proof.

\begin{proof}[Proof of Theorem \ref{thm:cs_generalization}]
Let $x\in\Sigma^N_k$ and let $T$ be the support of $x$. Without loss of generality, we assume that $\#T=k$. Denote the measurement matrix by $\Phi\in\mathbb{R}^{m\times N}$ whose rows are independent isotropic sub-Gaussian vectors whose sub-Gaussian norm does not exceed $K$. 

We use the two-stage reconstruction method. Accordingly, in Stage 1, we recover a coarse estimate $x^{\#}_{\rm{MSQ}}$of $x$ via
\begin{equation}
\label{eq:ell_1}
x^{\#}_{\rm{MSQ}}=\arg\min \|z\|_1\quad\quad \text{subject to } \|\Phi z-y\|\leq 0.5\delta\sqrt{m}.
\end{equation}
If $x$ is sufficiently sparse, $x^{\#}_{\rm{MSQ}}$ will be a good approximation of the original signal $x$. Indeed, if $\frac{1}{\sqrt{m}}\Phi$ satisfies the RIP of order $k$ with $\delta_{2k}<\sqrt{2}/2$, then \cite{CaiZhang2014}
$$\|x^{\#}_{\rm{MSQ}} - x\| \le C \delta.
$$
Recall that this is the case with overwhelming probability, for example, if $\Phi$ is an $m \times N$ sub-Gaussian matrix with independent isotropic rows whose $\psi_2$-norm does not exceed $K$ and $m\gtrsim_K k\log(N/k)$ -- see Corollary \ref{thm:cor:RIP_subGaussian}.

The coarse estimate $x^{\#}_{\rm{MSQ}} $ can be improved in the Stage 2 of the reconstruction. The key idea is that $T$, the support of $x$, can be recovered from $x^{\#}_{\rm{MSQ}}$ with high probability provided that all non-zero entries of $x$ exceed $\sqrt{2}C\delta$-- see \cite[Proposition 4.1]{Sobolev_Duals_for_RF}.

The next step is to recover the non-zero entries of $x$ using the support $T$, or, equivalently, to recover $x_T$. Note that $Q(y)=Q(\Phi x)=Q(\Phi_T x)$. Recall that rows of $\Phi$ are independent isotropic sub-Gaussian random vectors with sub-Gaussian norm at most $K$, and so are the rows of $\Phi_T$. Therefore, the conditions of Theorem \ref{thm:MSQforFrameGeneral} are satisfied for the frame $\Phi_T$.

In the setting of Theorem \ref{thm:MSQforFrameGeneral}, let $\alpha\in[0,1/2)$, $\lambda=m/k$, and let 
$c_2:=e^2\exp(-\lambda^{\alpha})$. Then, $c_1=C\sqrt{\ln(e^2c_2^{-1})}=C\lambda^{\alpha/2}$ and by 
Theorem \ref{thm:MSQforFrameGeneral}, for a given $x_T\in\mathbb{R}^k$, with probability exceeding ${1-2e^{-c_3m}-e^2\exp(-\lambda^{\alpha})}$, we have
\begin{equation}
\label{eq:MSQ_frame_to_CS}
\|\Phi_T^\dagger Q(\Phi_Tx_T)- x_T\|<A\left(2\|x\|\exp\left(-\frac{2\pi^2\|x\|^2}{\delta^2}\right)+c_7\sqrt{\log k}\lambda^{-1/2+\alpha/2}\delta\right),
\end{equation}
where $A=(1/2-c_K\lambda^{-1/2})^{-2}.$

To guarantee the fine recovery in the stage 2 of reconstruction, each $m\times k$ submatrix of $\Phi$ must satisfy the conditions of Theorem \ref{thm:MSQforFrameGeneral} which would imply the validity of \eqref{eq:MSQ_frame_to_CS}. In other words, we need to show that the rows of $\Phi_T$ are independent isotropic sub-Gaussian random vectors for all $\#T\leq k$. Recall that the rows of $m\times N$ matrix $\Phi$ are independent sub-Gaussian isotropic random vectors. Since $k$ rows of matrix $\Phi$ are chosen based on the support of the signal $x$, these rows are independent. Moreover, these rows remain isotropic and sub-Gaussian. Therefore, for each fixed choice of $k$ rows, \eqref{eq:MSQ_frame_to_CS} holds with probability at least $1-2\exp(-c_3m)-e^2e^{-\lambda^\alpha}.$ Finally, applying the result of Theorem \ref{thm:MSQforFrameGeneral} finishes the proof. 
\end{proof}

\subsection{Proof of Theorem \ref{thm:PBP_MSQ_bound:theorem}}
The proof loosely follows the spirit of \cite{xu2019quantized}. Let $S$ be the support of $x^{\text{PBP}}$ and $T$ be the support of $x$. Applying the triangle inequality, we get 
\begin{align*}
\|x^{\text{PBP}} - x\|& =\|(\frac{1}{m}\Phi^T Q(\Phi x))_S - x\|\\
& \leq \|(\frac{1}{m}\Phi^T Q(\Phi x))_S - (\frac{1}{m}\Phi^T \Phi x)_S\| + \|(\frac{1}{m}\Phi^T \Phi x)_S-(\frac{1}{m}\Phi^T \Phi x)_{S\cup T}\| + \|(\frac{1}{m}\Phi^T \Phi x)_{S\cup T} - x\|\\
&= \|(\frac{1}{m}\Phi^T Q(\Phi x))_S - (\frac{1}{m}\Phi^T \Phi x)_{S}\| + \|(\frac{1}{m}\Phi^T \Phi x)_{(S^c\cap T)}\| + \|(\frac{1}{m}\Phi^T \Phi x - x)_{S\cup T}\|\\
\end{align*}
Recall that $|S| =|T| = k$. Since $S$ is the index set of the $k$ largest entries of $\frac{1}{m}\Phi^T F(\Phi x)$ and $|S^c\cap T| = |S\cap T^c|$ we get the following bound:
\begin{align*}
\|x^{\text{PBP}} - x\|& \leq \|(\frac{1}{m}\Phi^T Q(\Phi x))_S - (\frac{1}{m}\Phi^T \Phi x)_{S}\| + \|(\frac{1}{m}\Phi^T \Phi x)_{(S\cap T^c)}\| + \|(\frac{1}{m}\Phi^T \Phi x - x)_{S\cup T}\|\\
&\leq \|(\frac{1}{m}\Phi^T F(\Phi x))_S\| + 2\|(\frac{1}{m}\Phi^T \Phi x - x)_{S\cup T}\|.\\
\end{align*}
Let $x_T\in\mathbb{R}^k$ be the vector that consists of non-zero elements of $x$. Note that the product $\Phi x$ can be considered as a linear combination of columns of $\Phi$, and only columns whose indexes are in $T$ are multiplied by a non-zero number. Therefore, $\Phi x = \Phi_T x_T$, where $\Phi_T$ is a submatrix of $\Phi$ that keeps only columns whose index is in $T$.

Applying the triangle inequality, we get
\begin{align*}
\|x^{\text{PBP}} - x\|&\leq\frac{1}{m}\|\left(\Phi^T F(\Phi x)\right)_S\| + 2\|(\frac{1}{m}\Phi^T \Phi x - x)_{S\cup T}\|\\
&\leq \frac{1}{m}\|\left(\Phi^T F(\Phi_T x_T)\right)_{S\cap T}\| +\frac{1}{m}\|\left(\Phi^T F(\Phi_T x_T)\right)_{S\cap T^c}\| + 2\|(\frac{1}{m}\Phi^T \Phi x - x)_{S\cup T}\|\\
&\leq \frac{1}{m}\|\Phi_{T}^T F(\Phi_T x_T)\| +\frac{1}{m}\|\Phi_{S\cap T^c}^T F(\Phi_T x_T)\| + 2\|(\frac{1}{m}\Phi^T \Phi x - x)_{S\cup T}\|.\\
\end{align*}
In the last line, we used that if we increase the number of entries from $|S\cap T|$ to $|T|$, the norm can only increase.

Let us bound each term individually. Note that rows of $\Phi_T$ keep all properties of rows of $\Phi$: they are independent isotropic sub-Gaussian with the same sub-Gaussian norm. Therefore, we can use Theorem \ref{thm:MSQforFrameGeneral} or Theorem \ref{thm:MSQForFrame} to estimate the first term.

We proceed to bounding the term $\frac{1}{m}\|\Phi_{S\cap T^c}^T F(\Phi_T x_T)\|$. Let $\Phi=(\phi)_{i,j}.$ For each $i\in S\cap T^c$, consider the $i$th entry of the vector $\Phi_{S\cap T^c}^T F(\Phi_T x_T)$:
$$\frac{1}{m}\left(\Phi_{S\cap T^c}^T F(\Phi_T x_T)\right)_i = \frac{1}{m}\sum_{j=1}^m \phi_{ji}F(\sum_{s\in T} \phi_{js}x_s).$$
Since $\Phi$ has independent rows, the expression above is the sum of independent random variables. Let us investigate each random variable. Since $S\cap T^c$ and $T$ do not intersect and the matrix $\Phi$ has independent rows, $\Phi_{S\cap T^c}^T$ and $\Phi_T$ are independent. Therefore, $\phi_{ji}$ and $F(\sum_{s\in T} \phi_{js}x_s)$ are independent. We conclude that $\mathbb{E}\phi_{ji}F(\sum_{s\in T} \phi_{js}x_s)=0$. Since $F(\Phi_T x_T)$ is bounded by $\delta/2$, the random variable $\phi_{ji}F(\sum_{s\in T} \phi_{js}x_s)$ is a sub-Gaussian random variable with mean 0 and sub-Gaussian norm at most $K\delta/2$. We can apply the Hoeffding inequality (Theorem \ref{thm:Hoeffding}) for bounding the sum of random variables. Taking the union bound for each entry and using $|S\cap T^c|\leq k$, we conclude
$$\|\frac{1}{m}\sum_{j=1}^m \phi_{ji}F(\sum_{s\in T} \phi_{js}x_s)\|\leq  C\sqrt{\log(k)}\lambda^{-1/2}\delta$$
with probability at least $1-e\exp(-C').$

The last term was bounded in Theorem \ref{thm:PBP_classical_result:theorem}. Combining all three bounds yields the statement of the theorem.

\section{Numerical experiments.}
\label{sec:numerical_experiments}
\begin{figure}[t]
  \centering
\begin{subfigure}{0.3\textwidth}
\includegraphics[height=0.2\textheight]{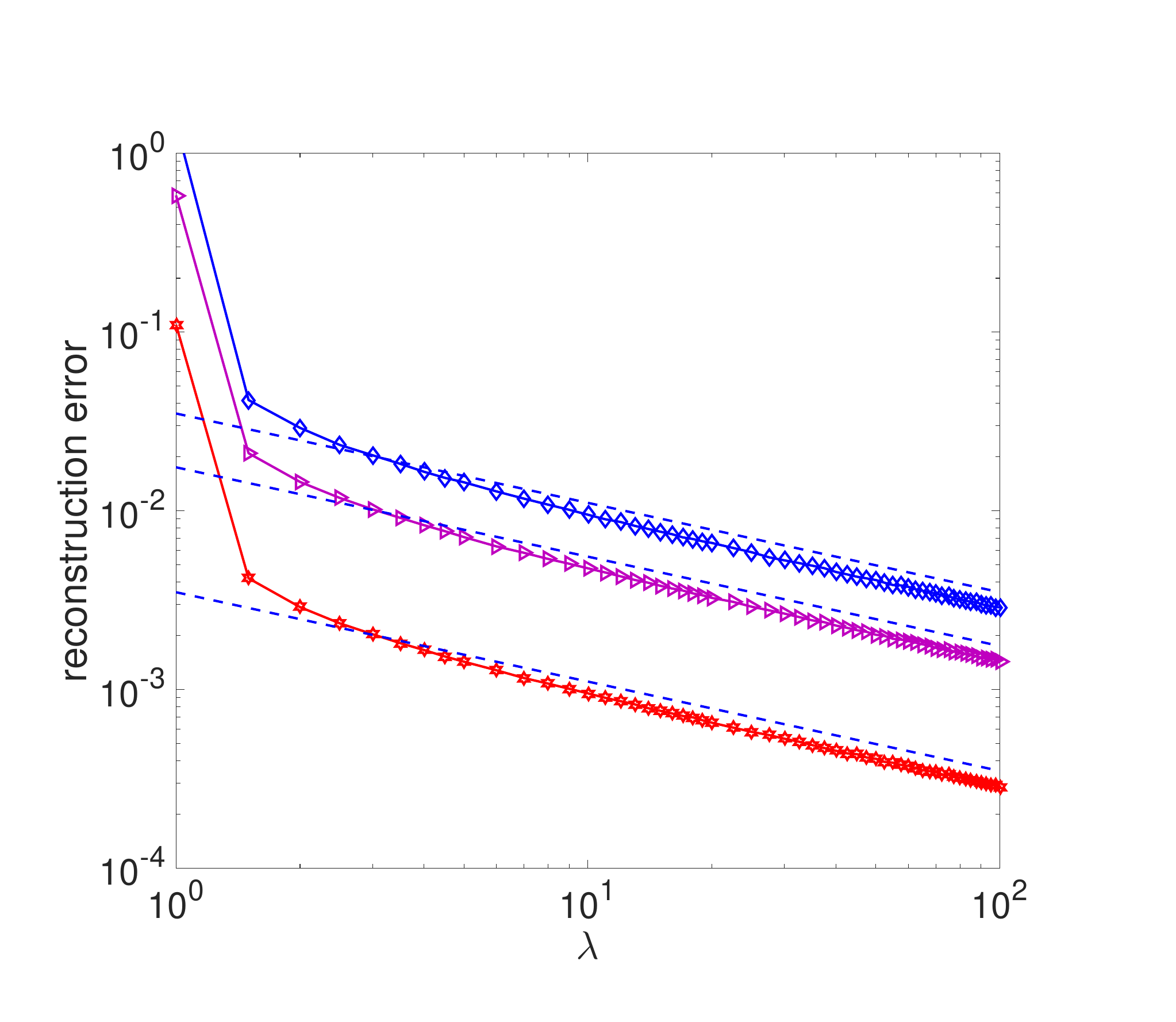} 
\vspace{-0.4cm}
\caption{ \label{fig:MSQ_WHN_Gaussian}}
\end{subfigure}
\hspace{5mm}
\begin{subfigure}{0.3\textwidth}
\includegraphics[height=0.2\textheight]{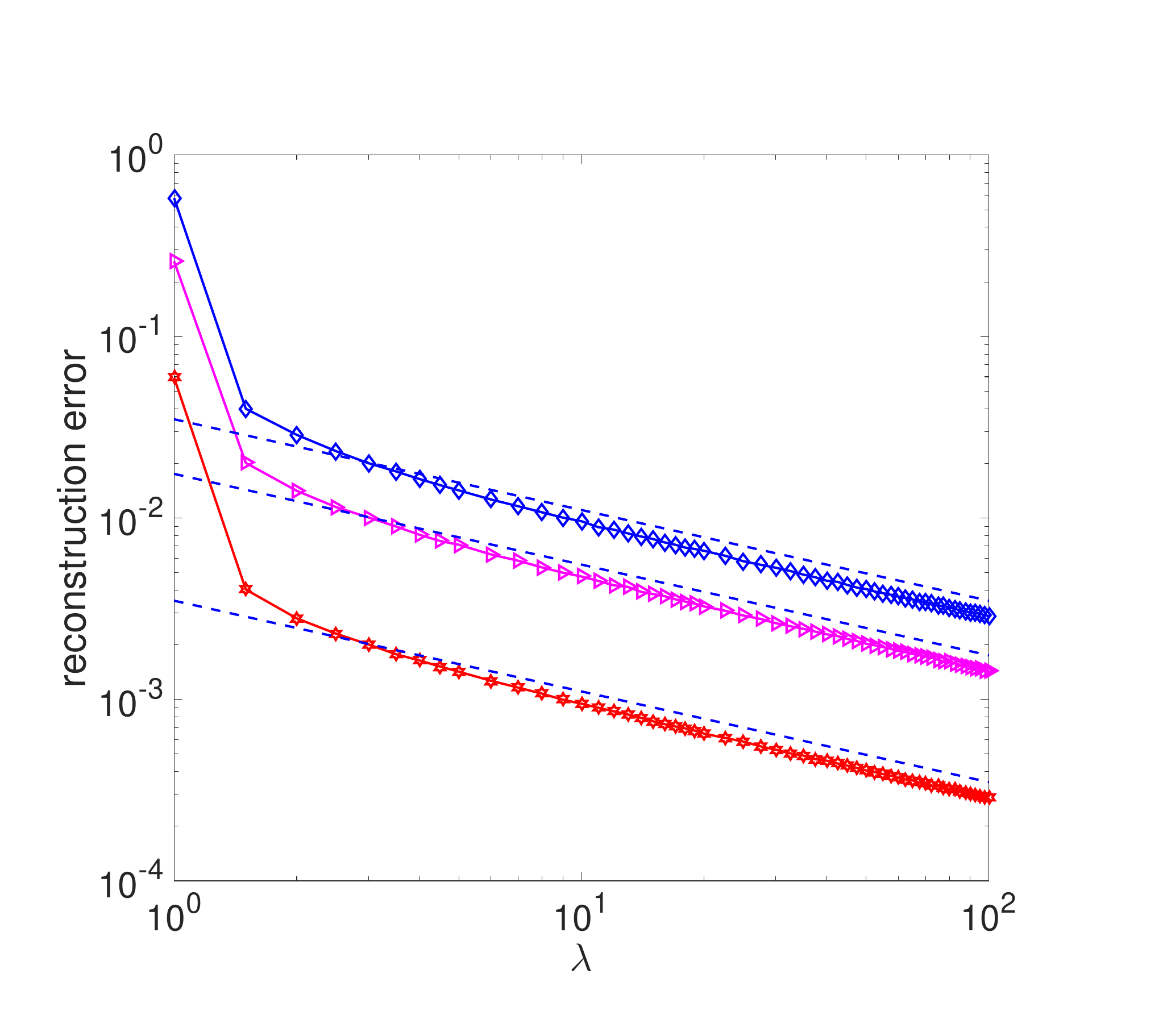}
\vspace{-0.4cm}
\caption{ \label{fig:MSQ_WNH_Bernoulli}}
\end{subfigure}
\hspace{4mm}
\begin{subfigure}{0.3\textwidth}
\includegraphics[height=0.2\textheight]{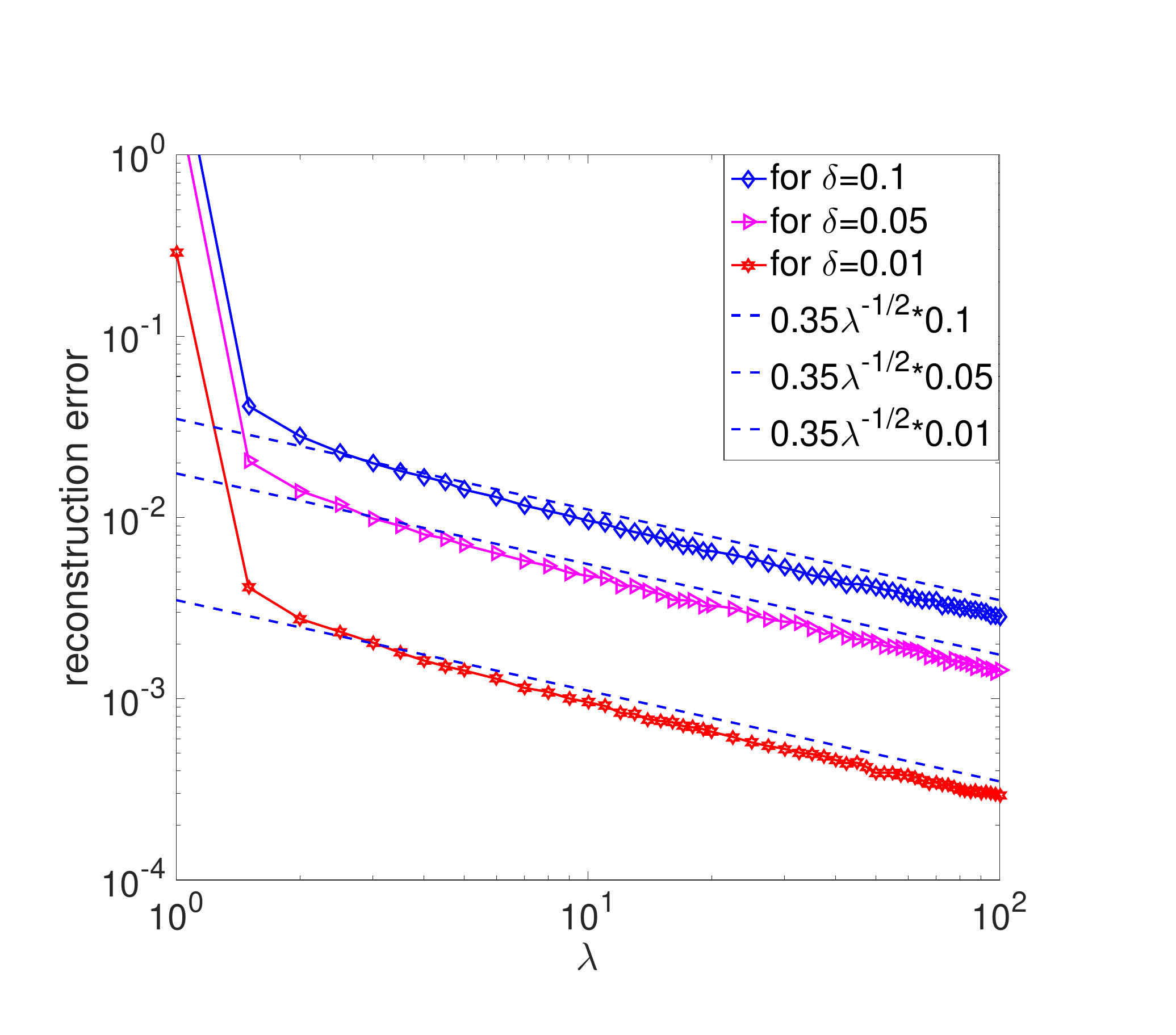}
\vspace{-0.4cm}
\caption{ \label{fig:MSQ_WNH_spherical}}
\end{subfigure}
  \caption{\label{fig:exp1_all} The numerical behavior (in log-log scale) of the mean of the reconstruction error, as a function of $\lambda=m/k$, for a fixed unit-norm signal $x$ and an $m\times k$ random matrix $E$ with i.i.d. standard Gaussian entries (A), with i.i.d Bernoulli entries (B), and with independent rows, uniformly distributed on the sphere of radius $\sqrt{k}$ (C). For each $\delta=0.1,\, 0.05,\, 0.01$, we draw 1000 realizations of the random matrix$E$, compute the quantized frame coefficients of $x$, and reconstruct using $E^\dagger$. We plot the average reconstruction error for each $\delta$.}
 \end{figure}

\noindent\textbf{Experiment 1. MSQ for frame theory.} This first set of experiments illustrates the error decay for the MSQ quantizer (with various quantizer stepsizes $\delta$)  for various classes of measurement matrices. 
To that end, we fix a unit-norm signal $x\in\mathbb{R}^k$ where $k=20$ and consider $m\times k$ random frames $E$, where $m$ varies between 20 and 2000 (i.e., $\lambda \in[1,100]$). For each $m$, we pick 1000 realizations of $E$ and calculate the reconstruction error $\mathcal{E}(x)=\|x-E^\dagger(Q^{\text{MSQ}}_\delta(Ex))\|$. We perform this for three distinct values of the quantizer resolution $\delta$, specifically $\delta\in \{0.01, 0.05, 0.1\}$. For each $\delta$, we report the average value of $\mathcal{E}(x)$ over the 1000 realizations of $E$. Figure \ref{fig:exp1_all} shows the outcomes in three different scenarios: when $E$ is a Gaussian matrix (A), a Bernoulli matrix (B), and when the rows of $E$ are drawn independently from the uniform distribution on the sphere with radius $k^{1/2}$ (C). The observed error decay rates seem to be consistent with m-WNH, which predicts a decay like $\delta \lambda^{-1/2}$, as well as with Theorem \ref{thm:MSQforFrameGeneral}, which shows that the error would behave like $C_1+C_2 \delta \lambda^{-1/2}$, where $C_1$ is a generally non-zero, but possibly very small constant, at least in the case of Gaussian random matrices. So, assuming that $C_1$ is small for the other random ensembles as well, the error in these experiments is far from saturating at the level of $C_1$. Next, we will design a numerical experiment that allows us to observe $C_1$.

\noindent\textbf{Experiment 2. Lower bound for the decay rate.} This experiment aims to provide numerical illustration of the result of Theorem \ref{thm:MSQforFrameGeneral}. Consider a deterministic unit-norm signal $x$ and random frame $E$ that satisfy the conditions of Theorem \ref{thm:MSQforFrameGeneral}. Then, Theorem \ref{thm:MSQforFrameGeneral} states that the reconstruction error $\mathcal{E}(x)$ satisfies 
$$A'\left(\mu-c_1K\sqrt{k}\lambda^{-1/2}\delta \right) \le \mathcal{E}(x)\leq A\left(\mu+c_1K\sqrt{k}\lambda^{-1/2}\delta \right)$$
with probability at least $1-c_2-2\exp(-c_3m)$ where $\mu=\frac{1}{m}\| \mathbb{E}E^T(Ex-Q(Ex))\|$, $A$ and $A'$ are some positive constants, $c_2\in(0,1)$, and $c_1=c_1(c_2)$ as in Theorem \ref{thm:MSQforFrameGeneral}. According to Proposition \ref{prop:constant_term_is_constant}, we expect the term $\mu$ to be of order $O(1)$ as $m\rightarrow\infty.$ In particular, if $\mu$ is not zero, the reconstruction error should tend to a value between $A'\mu$ and $A\mu$ as $\lambda=m/k\rightarrow\infty$. Note that this does not contradict with the outcomes of Experiment 1. For example, in the Gaussian case, in Theorem \ref{thm:estmean} we provide a rigorous estimate of the value of $\mu$, which is extremely small when $\delta<1$. Therefore, in the experiments with Gaussian frames, we can only observe the term $Ac_1K\sqrt{k}\lambda^{-1/2}\delta$ for the given range of $\lambda$.

In order to actually observe the influence of the constant term, we now set $\delta:=4$ (thus creating an artificial set-up) and repeat the Experiment 1 for Gaussian matrices, Bernoulli matrices, and matrices whose rows are randomly drawn from the uniform distribution on the sphere, with $\lambda \in [10,1000]$ this time. The outcomes are shown in Figure \ref{fig:art_all}. Specifically, in Figure \ref{fig:ConstanttermMSQ}, we observe that when $E$ is Gaussian, the error settles down to a value between ${2\|x\|(\exp(-2\pi^2\|x\|^2/\delta^2)-\exp(-8\pi^2\|x\|^2/\delta^2))}$ and ${2\|x\|\exp(-2\pi^2\|x\|^2/\delta^2)}$, as predicted. In Figure \ref{fig:ConstanttermMSQ_Bernoulli}, we observe that the behaviour of the error in the Bernoulli case is similar in that it settles to a value that is comparable with the Gaussian case. While we do not have a sharp estimate for the bounds in this case, the experiment suggests that, like the Gaussian case, $\mu$ is a non-zero, but small, constant when $E$ is a Bernoulli matrix.  Finally, Figure \ref{fig:ConstanttermMSQ_spherical} shows that a similar conclusion can be drawn for the case when $E$ is with random rows, drawn independently from the uniform distribution on the sphere with radius $k^{1/2}$.

\begin{figure}[t]
  \centering
\begin{subfigure}{0.3\textwidth}
\includegraphics[height=0.2\textheight]{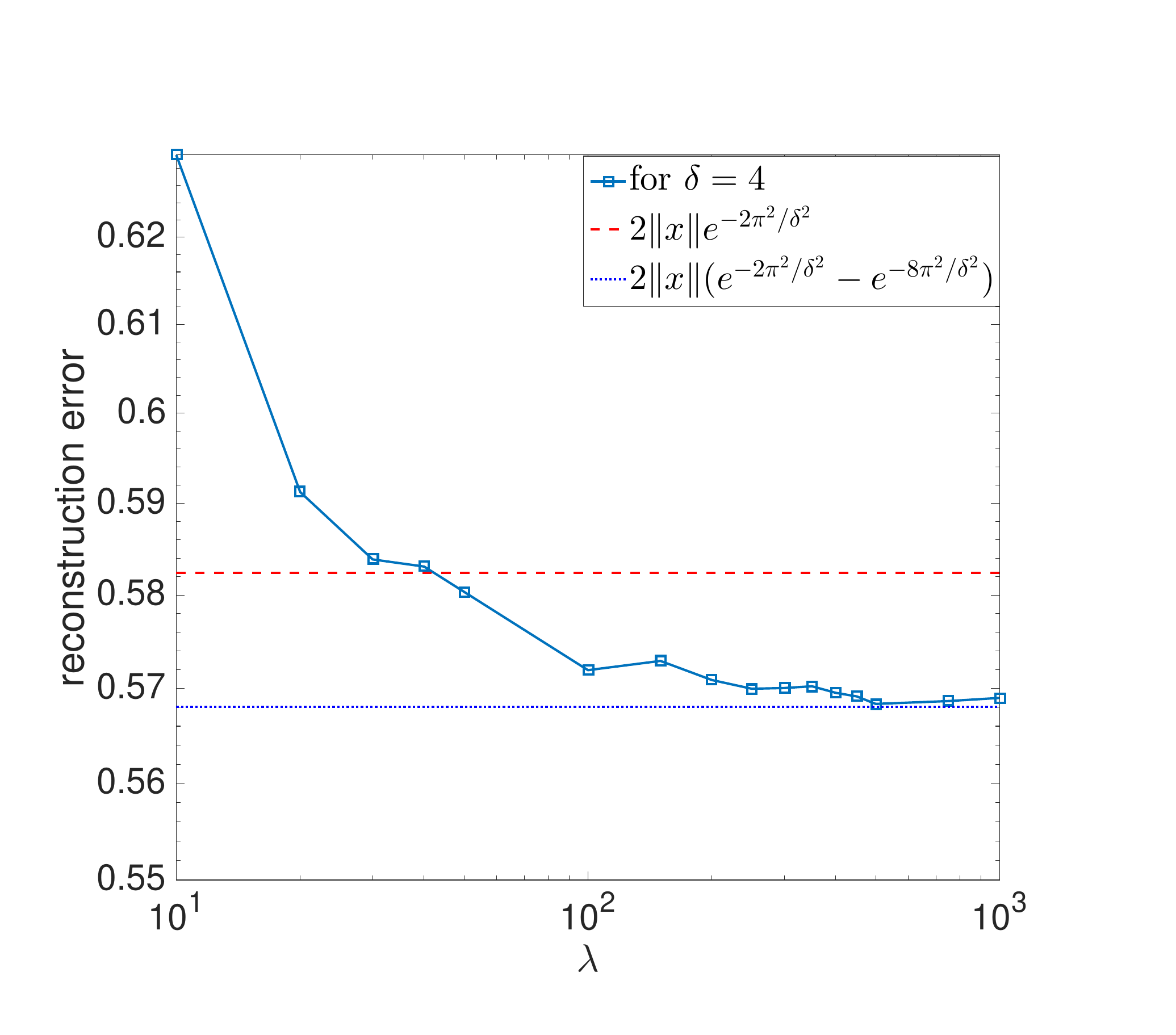} 
\vspace{-0.4cm}
\caption{ \label{fig:ConstanttermMSQ}}
\end{subfigure}
\hspace{5mm}
\begin{subfigure}{0.3\textwidth}
\includegraphics[height=0.2\textheight]{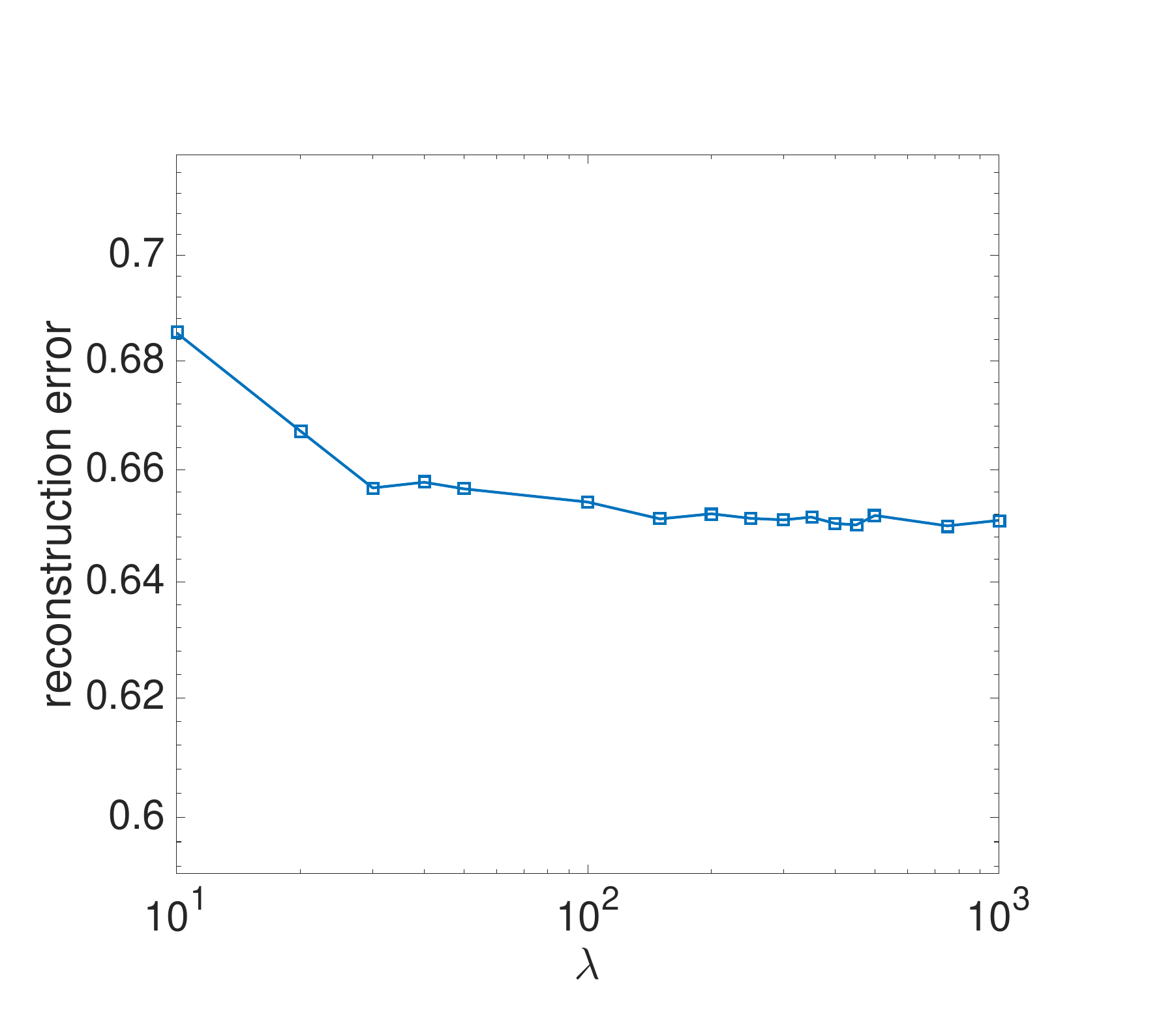}
\vspace{-0.4cm}
\caption{ \label{fig:ConstanttermMSQ_Bernoulli}}
\end{subfigure}
\hspace{4mm}
\begin{subfigure}{0.3\textwidth}
\includegraphics[height=0.2\textheight]{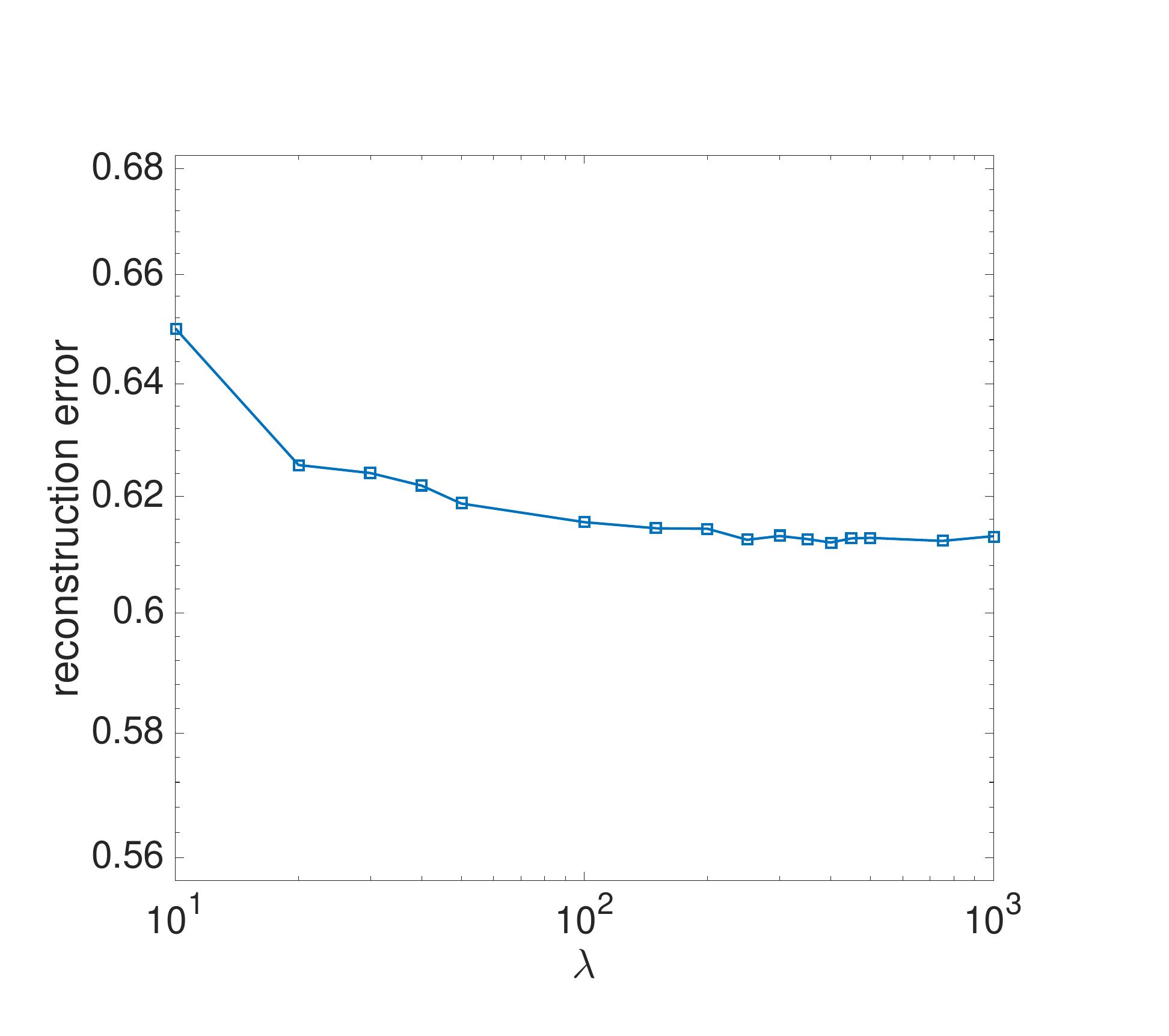}
\vspace{-0.4cm}
\caption{\label{fig:ConstanttermMSQ_spherical}}
\end{subfigure}
  \caption{\label{fig:art_all} The numerical behaviour of the reconstruction error in the setting identical to that described in Figure \ref{fig:exp1_all}; this time $\delta=4$ and $\lambda \in [10,1000]$. The results shown are the outcomes for Gaussian matrices (A), Bernoulli matrices (B), and matrices whose rows are randomly drawn from the uniform distribution on the sphere (C).   
      }
 \end{figure}

Finally, we consider random submatrices of the Discrete Fourier Transform (DFT) matrix whose rows and columns are drawn in the following way: We pick the first $k$ columns of the $N\times N$ Fourier matrix, with $N=100000$. Note that this results in a harmonic frame for $\C^k$. Next, we fix $m<N$; we randomly draw $m$ rows of this harmonic frame, resulting in an $m \times k$ random frame $E$. As above, for each $m$, we pick 1000 realizations of $E$ and calculate the average of the reconstruction error $\mathcal{E}(x)=\|x-E^\dagger(Q^{\text{MSQ}}_\delta(Ex))\|$ with $\delta=0.01$ over these realizations. The results are reported in Figure \ref{fig:MSQ_Fourier} which suggest that the average reconstruction error approaches a constant approximately equal to $\delta/2$. Note that this is qualitatively different from the other random matrix ensembles we considered in Experiment 1 in that we observe the constant term in the error bounds in a realistic setting (with $\delta=0.01$) rather than the artificial setting with $\delta=4$ above. 

\begin{figure}[t]
	\centerline{
		\includegraphics[height=0.18\textheight]{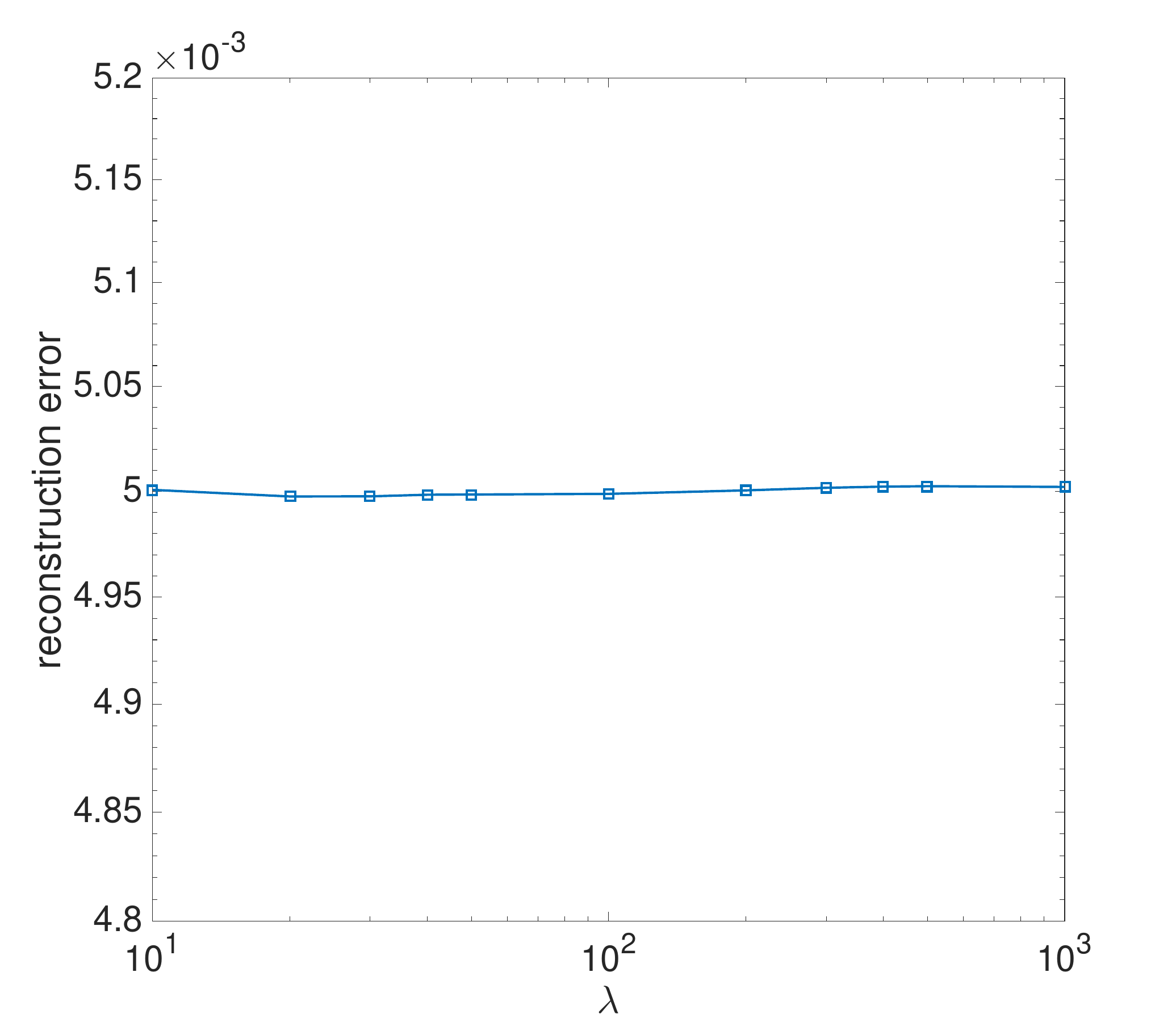}}
	\caption{The numerical behaviour of the reconstruction error in the setting identical to that described in Figure \ref{fig:exp1_all}; this time $E$ is an $m\times k$ random Fourier matrix obtained by restricting the $N\times N$ DFT matrix (with $N=100000$) to its first $k=20$ columns and then selecting $m$ random rows. Here $\delta=0.01$ and $\lambda \in [10,1000]$.}
	\label{fig:MSQ_Fourier}
\end{figure}

\textbf{Experiment 3. Compressed sensing setting.}
This experiments focuses on CS and is designed to observe the error decay shown in Theorem \ref{thm:cs_generalization} when the 2-stage reconstruction scheme described in Section \ref{sec:result_in_CS} is used to reconstruct a sparse $x$ from its quantized (by MSQ) compressed measurements. Let the signal $x\in\Sigma^{N}_{k}$, where $N=1000$ and $k=20$ and let $\Phi$ be $m\times N$. We will quantize compressed measurements of $x$, given by $\Phi x$, using MSQ. Here is a list of various parameters in our experiments: $\Phi$ will be Gaussian (Figure \ref{fig:MSQforCS}), Bernoulli (Figure \ref{fig:MSQforCS_Bernoulli}), or a random matrix whose rows are drawn i.i.d. from the sphere of radius $\sqrt{N}.$ (Figure \ref{fig:MSQforCS_spherical}). We will vary $m$ between 100 and 500, corresponding to $\lambda\in [5,25]$. The quantization stepsize is $\delta \in \{0.01,0.05,0.1\}$. Finally, to ensure the condition of Theorem \ref{thm:cs_generalization-gaussian} for successful support recovery is satisfied, we choose the non-zero entries of $x$ randomly from $\{\pm 1/\sqrt{k}\}$. It turns out that this condition is not satisfied when $\Phi$ is Bernoulli and $\delta=0.1$, so in the Bernoulli case we restrict our experiments to $\delta=0.01$ and $\delta=0.05$. The outcomes are given in Figure \ref{fig:cs_all}, where we observe that, in each case, the decay of the reconstruction error agrees with the result of Theorem \ref{thm:cs_generalization}.

\begin{figure}[h]
  \centering
\begin{subfigure}{0.3\textwidth}
\includegraphics[height=0.18\textheight]{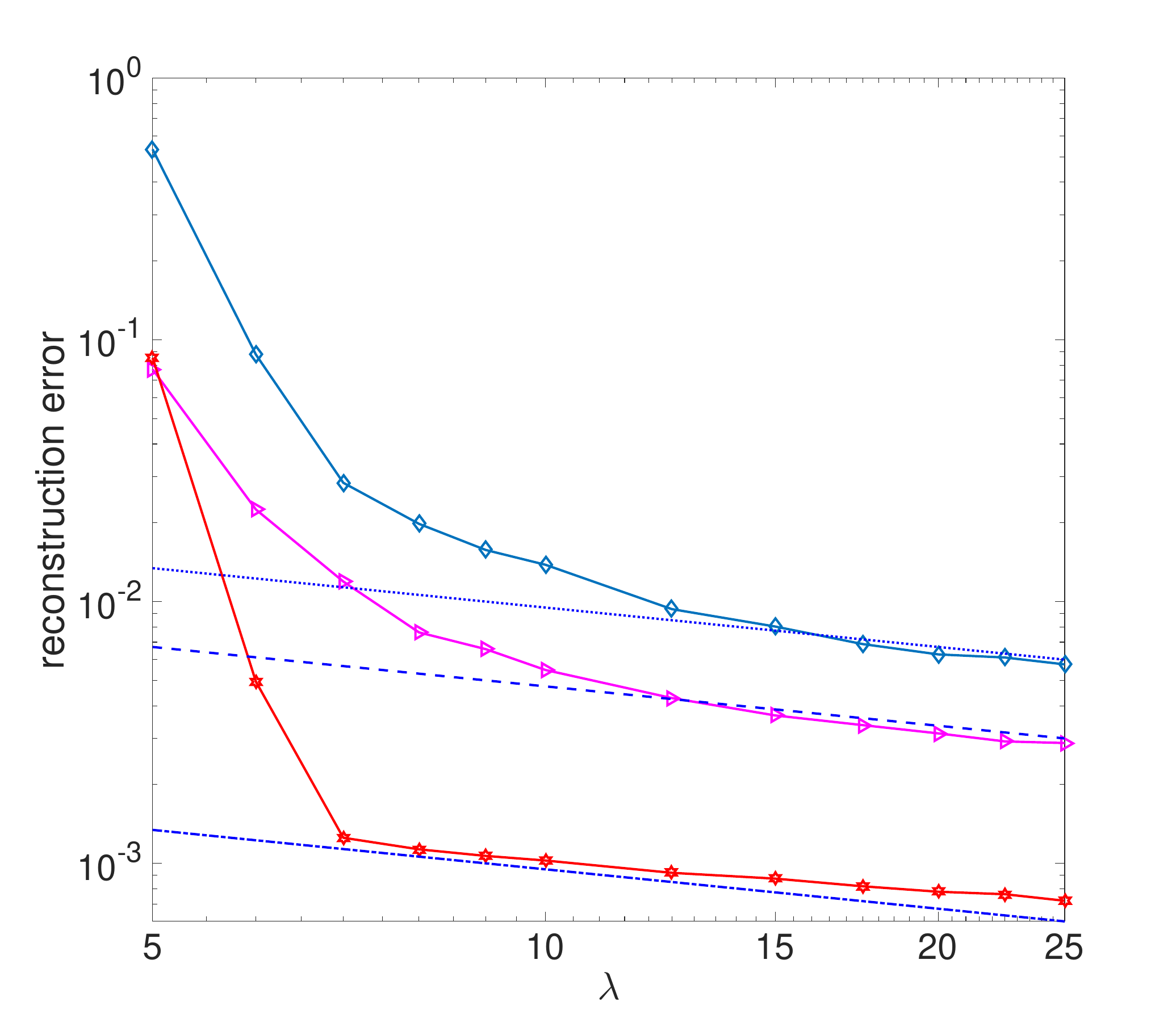} 
\vspace{-0.4cm}
\caption{ \label{fig:MSQforCS}}
\end{subfigure}
\hspace{5mm}
\begin{subfigure}{0.3\textwidth}
\includegraphics[height=0.18\textheight]{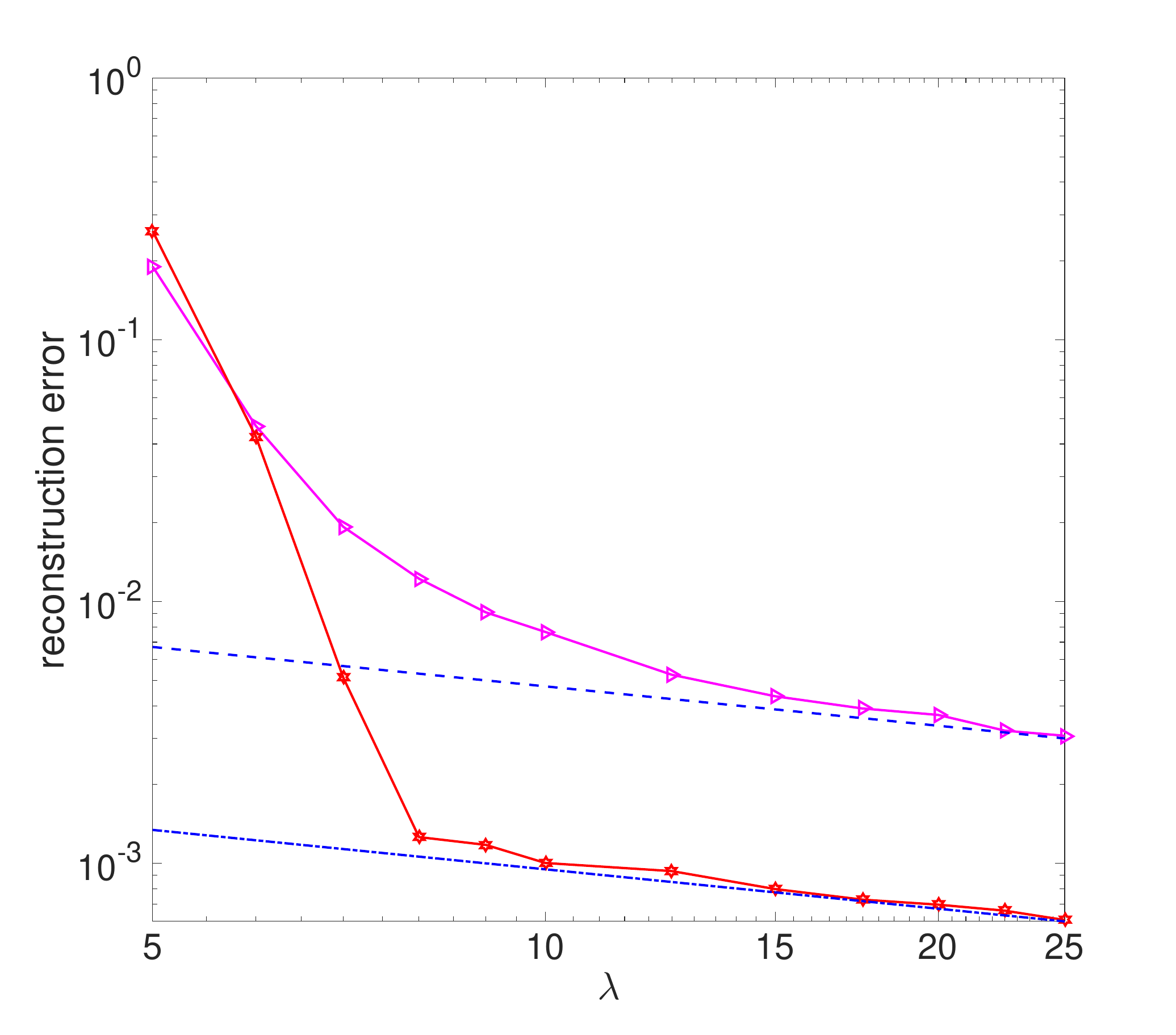}
\vspace{-0.4cm}
\caption{ \label{fig:MSQforCS_Bernoulli}}
\end{subfigure}
\hspace{4mm}
\begin{subfigure}{0.3\textwidth}
\includegraphics[height=0.18\textheight]{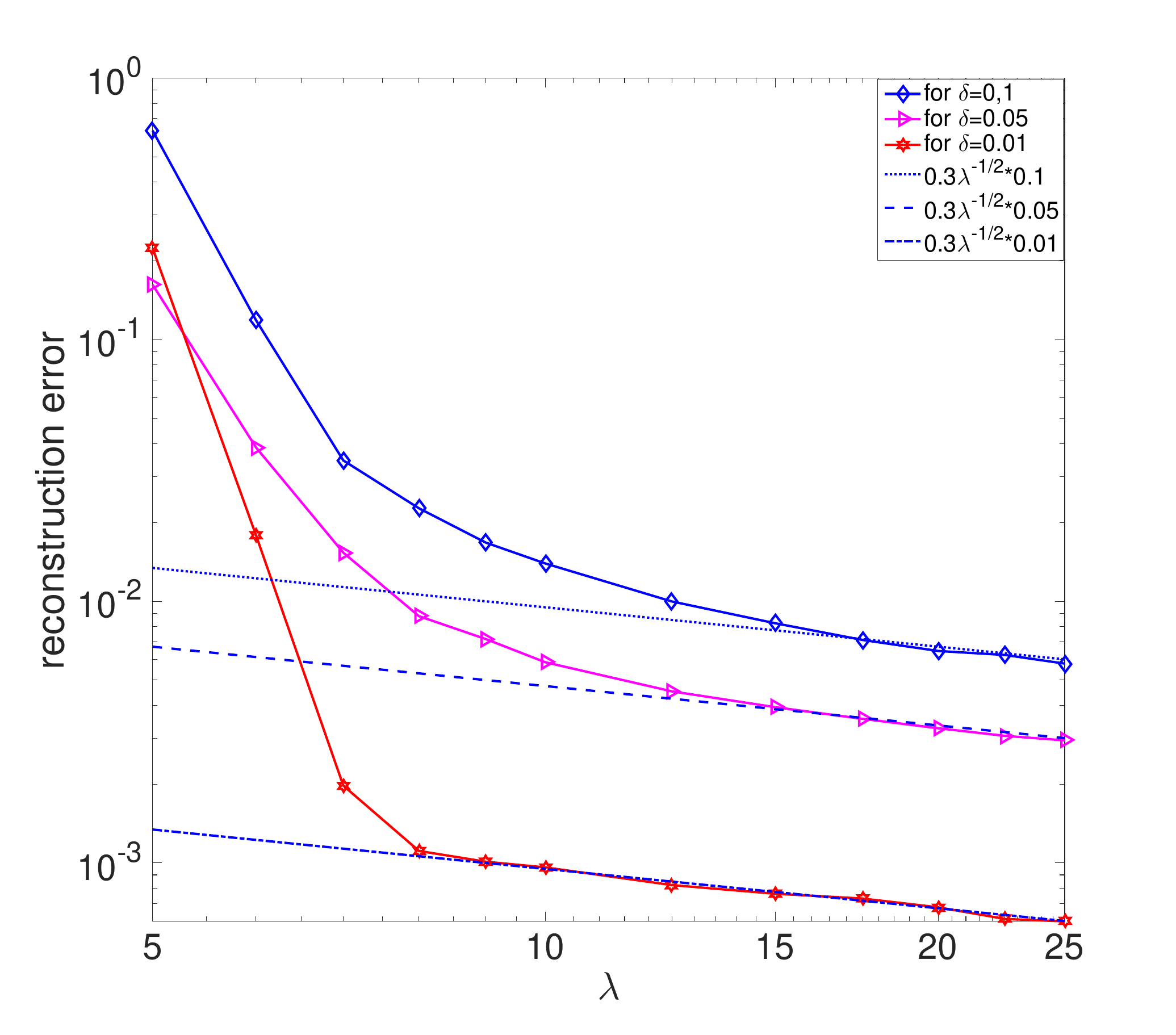}
\vspace{-0.4cm}
\caption{\label{fig:MSQforCS_spherical}}
\end{subfigure}
  \caption{\label{fig:cs_all} The numerical behaviour of the reconstruction error in the CS setting. The results shown are the outcomes for Gaussian measurement matrices (a), Bernoulli measurement matrices (b), and matrices whose rows are randomly drawn from the uniform distribution on the sphere (c).  Here $x\in\Sigma^{N}_{k}$, where $N=1000$ and $k=20$, the measurement matrix $\Phi$ is $m\times N$, where $m$ varies between 100 and 500, corresponding to $\lambda$ varying between 5 and 25. 
      }
 \end{figure}

\section{Acknowledgments} This work was funded in part by a UBC 4YF Fellowship (KM), an NSERC Discovery Accelerator Award (OY), and an NSERC Discovery Grant (OY).

\bibliographystyle{plain}
\bibliography{quantization,HA,CS,FrameTheory}

\end{document}